\documentclass[a4paper]{amsart}
\usepackage{amsmath,amssymb,amsthm}
\usepackage{color}
\usepackage[dvipsnames]{xcolor}
\usepackage{graphicx}
\usepackage{fullpage}
\usepackage{setspace}
\usepackage{tikz}
\usepackage{tikz-cd}
\usepackage[normalem]{ulem}
\usepackage{soul}
\usepackage{hyperref}

\numberwithin{equation}{section}
\newtheorem{theorem}{Theorem}[section]
\newtheorem{lemma}[theorem]{Lemma}

\theoremstyle{definition}
\newtheorem{definition}[theorem]{Definition}
\newtheorem{example}[theorem]{Example}

\newcommand{\cB}{\mathcal{B}}
\newcommand{\cP}{\ensuremath{\mathcal{P}}}

\newcommand{\dgm}{\mathrm{dgm}}
\newcommand{\Hom}{\mathrm{Hom}}
\newcommand{\interval}[1]{\lceil #1 \rfloor}
\newcommand{\barcode}{\mathrm{bar}}
\newcommand{\im}{\mathrm{im}\,}
\newcommand{\coker}{\mathrm{coker}\,}

\title{Matrix Algebra for Persistence Modules Yields a Proof of the Isometry Theorem}
\author{Michael Moy}
\date{}

\begin{document}

\begin{abstract}

The main purpose of this paper is to provide a simple proof of the isometry theorem for one-parameter persistence modules, in which a matrix representing one morphism of an interleaving is reduced and the pivots determine a matching between barcodes.
This approach applies to persistence modules of finite type indexed by the reals, and the more general statement for q-tame modules can then be deduced from it using approximations of q-tame modules.
The similarity between this use of matrix reduction and that in the persistent homology algorithm motivates some further development of matrix computations for persistence modules, formalized by a category of barcodes in which the morphisms are equivalence classes of matrices.
A method for computing induced maps on persistent homology is provided using matrix operations that fit naturally into persistent homology computations, making functorial persistent homology barcodes computable.
A method is also given for computing persistent homology barcodes and morphisms when chains do not necessarily have infinite death times.
\end{abstract}

\maketitle

\section{Introduction}\label{section:introduction}

In the study of one-parameter persistence, the isometry theorem states that the interleaving distance between q-tame persistence modules indexed by the reals is equal to the bottleneck distance between their persistence diagrams.
The theorem is typically split into two inequalities.
One states that the bottleneck distance is less than or equal to the interleaving distance: this is known as the algebraic stability theorem, as it is an algebraic generalization of an early version of the stability of persistent homology~\cite{cohen2007stability}.
The original proof of the algebraic stability theorem was given in~\cite{chazal_et_al_proximity_of_persistence_modules}, and a related approach is presented in a comprehensive study of persistence modules in~\cite{chazal2016structure}.
Algebraic stability has since served as a foundation on which other stability results can be based, including as a step in other versions of the stability of persistent homology~\cite{chazal2014persistence}.
The reverse inequality, sometimes known as converse algebraic stability, was then proved and subsequently simplified~\cite{lesnick2011optimality,Lesnick_2015_interleaving,chazal2016structure,bubenik2014categorification}, thus establishing the isometry theorem.
Contrary to the order in which these were originally proved, converse algebraic stability is now recognized as the easier direction, so most of the work required to establish the isometry theorem lies in proving the algebraic stability theorem.
Alternate proofs of the algebraic stability theorem with contrasting approaches have been given, including in~\cite{Bauer_Lesnick_2015, bjerkevik_stability}.
A more thorough discussion of the history and various approaches to the isometry theorem are given in~\cite{Bauer_Lesnick_2015}, although this predates~\cite{bjerkevik_stability}.

This paper presents a proof of the isometry theorem that takes a matrix-based approach to persistence modules of finite type, then generalizes to q-tame modules.
The method for modules of finite type, presented in Section~\ref{section: isometry theorem finite type}, is particularly simple once the right algebraic setup has been made: after reducing a matrix representing one morphism of an $\varepsilon$-interleaving, the pivots determine an $\varepsilon$-matching.
The matrix-based approach places the proof of the isometry theorem in the same language as the persistent homology algorithm, and I believe this unifying language will have a pedagogical benefit, allowing the persistent homology algorithm and stability to be more easily learned side by side.
This approach has been partly inspired by other proofs of the isometry theorem, and their influence should be evident to those familiar with them.
The setup of the proof is related to the method of~\cite{bjerkevik_stability}, which uses ranks of matrices to bound numbers of matchable bars, then proves the existence of a matching by a combinatorial argument.
The use of the underlying algebraic structure of persistence modules and duality is similar to the method of~\cite{Bauer_Lesnick_2015}, and the resulting matchings, though not identical, are similar in that a matching is constructed from one morphism of an interleaving.

After proving the isometry theorem for modules of finite type, the generalization to q-tame modules is given in Sections~\ref{section: isometry theorem locally finite} and~\ref{section: isometry theorem q-tame}.
These sections follow a natural progression, first extending to locally finite modules by approximating them on bounded intervals, and then to q-tame modules, which are approximated in both the interleaving and bottleneck distances by locally finite modules.
I previously wrote about these methods in~\cite{moy2024persistence}, and these sections are essentially independent of the rest of the paper.

The use of matrices in the proof of the isometry theorem and the similarity with the persistent homology algorithm has motivated some additional algebraic work in this paper.
The algebraic approach will seem familiar to those used to viewing persistence modules indexed by $\mathbb{N}$ as graded modules over $F[t]$ (as in~\cite{zomorodian_carlsson_computing}, for instance).
With a slight generalization, matrices can be used to represent morphisms of appropriate classes of persistence modules indexed by any totally ordered set $R$.
The approach used here does not require any notion of an underlying ring for the module and instead relies only on matrices of field elements and the associated birth and death times of bars.
Categorically speaking, this means that the equivalence of categories of persistence modules of finite type indexed by $\mathbb{N}$ and finitely generated graded modules over $F[t]$ can be generalized to a variant designed around one-parameter persistence: the category of persistence modules of finite type indexed by $R$ is equivalent to a category in which objects are barcodes over $R$ and morphisms are certain equivalence classes of matrices.
In addition to the computational convenience of matrices, persistent homology barcodes are functorial in this context, and I hope this perspective will be explored further.

The algebraic approach here is therefore not entirely novel, but as it is also not commonplace, self-contained derivations of the matrix techniques are given.
These can be compared and contrasted with computational techniques for graded modules, described in the context of persistence in~\cite{skraba2013persistence}, for instance.
The definitions, first properties, and general computational approach are given in Sections~\ref{section:matrices represent morphisms} and~\ref{section:matrix operations for persistence modules}.
In order to expedite the proof of the isometry theorem, which only requires these basic techniques, further properties are deferred until after the sections on the isometry theorem.
Section~\ref{section: a category of barcodes} gives the categorical formalization outlined above, and Section~\ref{section: computations} gives matrix-based methods to compute the image, kernel, and cokernel of a morphism of persistence modules. 
The computations are more involved than those for vector spaces, but they ultimately consist of row and column operations.

To demonstrate the matrix techniques and explain the similarities between the isometry theorem and the persistent homology algorithm, Section~\ref{section: functorial PH algorithm} focuses on matrix methods for computing persistent homology.
A method for computing induced morphisms of persistent homology modules is given, making the ``functorial barcode'' perspective of Section~\ref{section: a category of barcodes} computational, and the kernel and cokernel computations of Section~\ref{section: computations} are applied to compute persistent homology from chain complexes in which chains do not necessarily have infinite death times.
This is related to some previous work, although no detailed comparison will be made here.
For instance, methods for computing persistent homology in the same setting of general death times have previously been given in~\cite{skraba2013persistence} and ~\cite{dey2024efficient} using presentations of graded modules; the underlying computations are related to those here, although they are presented in different language.
An algorithm for filtrations of simplicial complexes that are not necessarily increasing has also been given in~\cite{dey2014computing}.
Further work will hopefully be done on efficient implementations of the methods presented here for practical use in code for computing persistent homology.

\section{Background}\label{section:background}

Here we quickly review some definitions used in the study of one-parameter persistence modules.
Readers familiar with persistence can likely skip this section and refer to it as needed.
More detailed introductions can be found in~\cite{chazal2016structure,oudot_persistence}, for instance.

\subsection{Persistence modules}\label{subsection:background, persistence modules}

We fix a field $F$ throughout and work with vector spaces over $F$.
A \emph{persistence module} $V$ indexed by a totally ordered set $R$ consists of a vector space $V_t$ for each $t \in R$ and a linear map $V_{s \leq t} \colon V_s \to V_t$ for each pair of elements $s \leq t$, such that $V_{r \leq t} = V_{s \leq t} \circ V_{r \leq s}$ whenever $r \leq s \leq t$.
The maps $V_{s \leq t}$ are called the \emph{structure maps} of $V$.
A motivating special case uses $R = \mathbb{R}$, with the usual order, but much of the theory applies in general.
A \emph{morphism} $\varphi \colon U \to V$ of persistence modules, both indexed by $R$, consists of a linear map $\varphi_t$ for each $t \in R$, such that for any $s \leq t$, the following diagram commutes.
\[
\begin{tikzcd}
U_s \arrow[r, "U_{s \leq t}"] \arrow[d, "\varphi_s"'] & U_t \arrow[d, "\varphi_t"] \\
V_s \arrow[r, "V_{s \leq t}"']                        & V_t                       
\end{tikzcd}
\]
In short, the category of persistence modules indexed by $R$ is the functor category $\mathrm{Vect}^R$.
Zero modules, direct sums, quotients, and the images, kernels, and cokernels of morphisms are all computed pointwise. 
For instance, with $\varphi$ as above, $(\ker \varphi)_t = \ker \varphi_t$ and $(\ker \varphi)_{s \leq t}$ is the restriction of $U_{s \leq t}$.

Given an interval $L$ in $R$, the \emph{interval module} $\mathbb{I}(L)$ is a persistence module defined by letting
\[
\mathbb{I}(L)_t = 
\begin{cases}
    F & \text{ if $t \in L$}\\
    0 & \text{ if $t \notin L$}
\end{cases}
\]
and letting $\mathbb{I}(L)_{s \leq t}$ be the identity map if $s,t \in L$ and the zero map otherwise.
It can be shown that an interval module is indecomposable, i.e., it cannot be expressed as a direct sum of two nonzero modules.
On the other hand, many persistence modules can be expressed as direct sums of interval modules (Theorem~\ref{theorem: pfd implies interval decomposable}, below). 
A persistence module $V$ is called \emph{interval decomposable} if it is isomorphic to a direct sum of interval modules: $V \cong \bigoplus_{i \in I} \mathbb{I}(L_i)$.

\subsection{Finiteness conditions}\label{subsection:background, finiteness conditions}

Various finiteness conditions are used in the study of persistence modules.
A persistence module $V$ indexed by $R$ is called \emph{pointwise finite dimensional} if $V_t$ is finite dimensional for every $t \in R$.
A persistence module $V$ is called \emph{q-tame} if $V_{s \leq t}$ has finite rank whenever $s < t$ (this condition is primarily applied to persistence modules indexed by $\mathbb{R}$).
An interval-decomposable persistence module is of \emph{finite type} if it is isomorphic to a finite direct sum of interval modules.
An interval-decomposable persistence module $V \cong \bigoplus_{i \in I} \mathbb{I}(L_i)$ indexed by $\mathbb{R}$ is \emph{locally finite} if for each $t \in \mathbb{R}$, there is a neighborhood of $t$ that intersects finitely many of the intervals $L_i$.
By compactness of closed bounded intervals in $\mathbb{R}$, $V$ is locally finite if and only if any bounded interval intersects finitely many $L_i$.

\subsection{Decorated reals}\label{subsection:background, decorated reals}

Intervals in $\mathbb{R}$ can be described using the set of \emph{decorated real numbers} $\mathbb{R} \times \{-,+\}$, defined for instance in~\cite{chazal2016structure}.
For real numbers $s \leq t$, intervals in $\mathbb{R}$ are written as shown below.
\begin{align*}
    &\interval{-\infty, \infty} = (-\infty, \infty)  & & \interval{s^-,\infty} = [s,\infty)  & &\interval{s^+,\infty} = (s,\infty) \\
    &\interval{-\infty, t^+} = (-\infty, t]          & &\interval{s^-,t^+} = [s,t]           & &\interval{s^+,t^+} = (s,t]         \\
    &\interval{-\infty, t^-} = (-\infty, t)          & & \interval{s^-,t^-} = [s,t)          & &\interval{s^+,t^-} = (s,t) 
\end{align*}
The decorated reals are given the lexicographic order using the usual order on $\mathbb{R}$ and the order $- < +$ on decorations.
For $s,t \in \mathbb{R}$, we let reals be added to decorated reals by setting $s^\pm + t = (s+t)^\pm$, and we use the convention $\pm \infty + t = \pm \infty$.
Intervals maybe written as $L = \interval{b,d}$, where $b,d \in \mathbb{R} \times\{+,-\} \cup \{-\infty, +\infty\}$ are referred to as birth and death times, and we let $L+t = \interval{b+t, d+t}$.

\subsection{Barcodes and persistence diagrams}\label{subsection:background, barcodes}

Since there may be multiple copies of the same interval module in a direct sum defining an interval-decomposable module, the intervals are usually described in the language of multisets.
We will frequently index a multiset by a set, so that elements of the multiset can be identified by their indices.
Formally, a multiset of intervals of $R$ indexed by $I$ is a function from $I$ to the set of intervals of $R$, which we write as $\{L_i\}_{i \in I}$.
Then $\{L_i\}_{i \in I}$ and $\{L'_j\}_{j \in J}$ represent the same multiset if and only if there exists a bijection $f \colon I \to J$ such that $L'_{f(i)} = L_i$ for all $i \in I$.

If a persistence module is interval decomposable, then the multiset of intervals in its decomposition is unique.
Uniqueness of interval decomposition can be proved using a Krull-Schmidt theorem (see~\cite[Theorem~1.3 (Krull--Remak--Schmidt--Azumaya)]{chazal2016structure}), and it can also be proved directly with reasoning similar to the proof of Lemma~\ref{lemma:morphisms_of_interval_modules} below; see~\cite[Theorem~2.2.3]{moy2024persistence}.
However, there can be multiple isomorphisms identifying the module with the same direct sum of interval modules, analogous to different choices of a basis of a vector space.
If $V \cong \bigoplus_{i \in I} \mathbb{I}(L_i)$, the multiset $\{L_i\}_{i \in I}$ is called the \emph{barcode} of $V$, written $\barcode(V)$, and the intervals $L_i$ may be referred to as bars.
Existence of interval decompositions, and thus of barcodes, is guaranteed in a general setting by the following theorem.

\begin{theorem}[Crawley-Boevey~\cite{crawley_boevey_persistence_modules}]\label{theorem: pfd implies interval decomposable}
    If $R$ is a totally ordered set with a countable dense subset in the order topology, then any pointwise finite-dimensional persistence module indexed by $R$ is interval decomposable.
\end{theorem}

Let $V \cong \bigoplus_{i \in I} \mathbb{I}\interval{b_i, d_i}$ be a persistence module indexed by $\mathbb{R}$, and for each $i$, let $\overline{b_i}$ and $\overline{d_i}$ be elements of the extended reals, obtained by removing decorations from $b_i$ and $d_i$.
The (undecorated) \emph{persistence diagram} $\dgm(V)$ is defined to be the multiset of ordered pairs $(\overline{b_i}, \overline{d_i})$ for all $i \in I$ such that $\overline{b_i} \neq \overline{d_i}$.
It is a subset of the extended half plane $\{(x,y) \in (\mathbb{R}\cup\{\pm \infty\})^2 \mid x<y\}$.
This definition of the persistence diagram applies to any interval-decomposable persistence module indexed by $\mathbb{R}$, and it is also common to extend the definition to q-tame modules; we will address this in Section~\ref{section: isometry theorem q-tame}.

\subsection{Distances}\label{subsection:background, distances}

Persistence modules indexed by $\mathbb{R}$ are compared by ``approximate isomorphisms'' called \emph{interleavings}.
If $V$ is a persistence module indexed by $\mathbb{R}$, then for any $\varepsilon \in \mathbb{R}$ there is a shifted module $V_{\_ + \varepsilon}$ with $t$ component $V_{t + \varepsilon}$ and with structure map $V_{s + \varepsilon \leq t + \varepsilon}$ for each inequality $s \leq t$.
For $\varepsilon \geq 0$, we have a morphism $V \to V_{\_ + \varepsilon}$ with $t$ component $V_{t \leq t+ \varepsilon}$, which we refer to as a shift morphism.
For any $\varepsilon \geq 0$, an \emph{$\varepsilon$-interleaving} between persistence modules $U$ and $V$ is a pair of morphisms $\varphi \colon U \to V_{\_ + \varepsilon}$ and $\psi \colon V \to U_{\_ + \varepsilon}$ such that the following diagrams commute for all $t \in \mathbb{R}$.
\[
\begin{tikzcd}
U_t \arrow[rd, "\varphi_t"'] \arrow[rr, "U_{t \leq t+2 \varepsilon}"] &                                                          & U_{t + 2 \varepsilon} &                                                                     & U_{t + \varepsilon} \arrow[rd, "\varphi_{t+\varepsilon}"] &                       \\
                                                                      & V_{t + \varepsilon} \arrow[ru, "\psi_{t+ \varepsilon}"'] &                       & V_t \arrow[rr, "V_{t \leq t+ 2 \varepsilon}"'] \arrow[ru, "\psi_t"] &                                                           & V_{t + 2 \varepsilon}
\end{tikzcd}
\]
Equivalently, $\psi_{\_ + \varepsilon} \circ \varphi$ is the shift map $U \to U_{\_ + 2 \varepsilon}$ and $\varphi_{\_ + \varepsilon} \circ \psi$ is the shift map $V \to V_{\_ + 2 \varepsilon}$, where $\varphi_{\_ + \varepsilon} \colon U_{\_ + \varepsilon} \to V_{\_ + 2 \varepsilon}$ and $\psi_{\_ + \varepsilon} \colon V_{\_ + \varepsilon} \to U_{\_ + 2 \varepsilon}$ are shifted versions of $\varphi$ and $\psi$.
The \emph{interleaving distance} between two persistence modules indexed by $\mathbb{R}$ is defined by
\[
d_I(U,V) = \inf \{ \varepsilon>0 \mid \text{there exists an $\varepsilon$-interleaving between $U$ and $V$} \}.
\]
The interleaving distance satisfies the triangle inequality, which can be seen by appropriately composing interleavings.

Barcodes of persistence modules indexed by $\mathbb{R}$ are compared by testing how closely their intervals match.
Given $\barcode(V) = \{\interval{b_i,d_i}\}_{i \in I}$ and $\barcode(U) = \{\interval{b'_j, d'_j}\}_{j \in J}$ with fixed indexing as shown, we define a \emph{matching} of the barcodes to be a subset $M \subseteq I \times J$ such that for each $i \in I$ there is at most one $j$ such that $(i,j) \in M$, and similarly for each $j \in J$ there is at most one $i \in I$ such that $(i,j) \in M$.
If $(i,j) \in M$, we will say $\interval{b_i,d_i}$ and $\interval{b'_j,d'_j}$ are matched by $M$.
We will say $M$ is an \emph{$\varepsilon$-matching} if
\begin{enumerate}
    \item every bar in both $\barcode(U)$ and $\barcode(V)$ that contains a closed interval of length $2 \varepsilon$ is matched by $M$ to a bar in the other barcode
    \item for all $(i,j) \in M$, we have $\interval{b_i, d_i} \subseteq \interval{b'_j - \varepsilon, d'_j + \varepsilon}$ and $\interval{b'_j, d'_j} \subseteq \interval{b_i - \varepsilon, d_i + \varepsilon}$.
\end{enumerate}
The \emph{bottleneck distance} between two barcodes is defined by
\[
d_B( \barcode(U), \barcode(V)) = \inf \{ \varepsilon>0 \mid \text{there exists an $\varepsilon$-matching between $\barcode(U)$ and $\barcode(V)$}\}.
\]

The definition of $\varepsilon$-matchings of barcodes used here is equivalent to that in~\cite{Bauer_Lesnick_2015}. 
This is slightly stricter than the definition of an $\varepsilon$-matchings of persistence diagrams in~\cite{chazal2016structure}, which requires 1) that all points $(\overline{b}, \overline{d})$ in either diagram such that $\overline{d} - \overline{b} > 2\varepsilon$ are matched and 2) that $|\overline{b}_1 - \overline{b}_2| \leq \varepsilon$ and $|\overline{d}_1 - \overline{d}_2| \leq \varepsilon$ if $(\overline{b}_1, \overline{d}_1)$ and $(\overline{b}_2, \overline{d}_2)$ are matched.
The bottleneck distance is defined analogously for diagrams and also written as $d_B$.
These definitions of $\varepsilon$-matchings for barcodes and diagrams differ only in the treatment of open and closed endpoints of bars, so the bottleneck distance is not affected by this difference.
In either case, the bottleneck distance satisfies the triangle inequality, which can be checked by composing matchings (following the general definition of composition of relations).

\section{Matrices Represent Morphisms}\label{section:matrices represent morphisms}

One-parameter persistence modules can informally be thought of as vector spaces evolving over time.
With reasonable assumptions, this analogy extends to bases, so that persistence modules can be completely described by basis elements with lifetimes.
While there can be some choice in the basis elements, the collection of lifetimes is invariant: this is the barcode of the persistence module, analogous to the dimension of a vector space.
To continue the analogy, we should expect a morphism to be expressible in terms of these basis elements with lifetimes and thus described by a matrix.
This perspective is most often taken for persistence modules of finite type indexed by the natural numbers, and it follows that it will also apply to persistence modules of finite type indexed by any totally ordered set.

Here we describe general matrix representations of morphisms of persistence modules of finite type\footnote{Generalizations to infinite cases are briefly considered in Section~\ref{subsection: infinite barcodes}.}.
We will fix a field $F$ and work with vectors spaces over $F$ throughout.
In general, we let persistence modules be indexed by a totally ordered set $R$, where the main case of interest and primary example is $R=\mathbb{R}$.
Readers used to viewing persistence modules indexed by $\mathbb{N}$ as graded modules over $F[t]$ should note that the techniques here are similar to matrix computations with graded modules; the approach here makes no mention of an underlying ring and is instead centered around the total order of $R$.

In order to handle the general index set $R$ while keeping $\mathbb{R}$ at the forefront, we use the following notation for intervals in $R$ that agrees with the use of decorated real numbers to describe intervals in $\mathbb{R}$.
A \emph{cut} $c$ of $R$ is an ordered pair $(c^\downarrow, c^\uparrow)$ of subsets (allowed to be empty) such that $c^\downarrow \cup c^\uparrow = R$ and for all $r_1 \in c^\downarrow$ and $r_2 \in c^\uparrow$, we have $r_1 < r_2$.
Cuts are ordered by inclusion of their lower sets, that is, $c_1 \leq c_2$ if $c_1^\downarrow \subseteq c_2^\downarrow$.
Given cuts $b$ and $d$ of $R$ with $b\leq d$, we define $\interval{b,d} = b^\uparrow \cap d^\downarrow$ and refer to $b$ and $d$ as birth and death times.
We also let $-\infty$ and $\infty$ denote the cuts $(\varnothing, R)$ and $(R, \varnothing)$, so that $\interval{-\infty,d} = d^\downarrow$ and $\interval{b,\infty} = b^\uparrow$.
Any nonempty interval $L \subseteq R$ can be written uniquely as $L = \interval{b,d}$.

We begin by examining the possible morphisms between interval modules.
Let $L = \interval{b,d}$ and $L' = \interval{b', d'}$, and suppose $\varphi \colon \mathbb{I}(L') \to \mathbb{I}(L)$ is a morphism.
A component $\varphi_{t}$ is zero if $t \notin L \cap L'$, and if $t \in L \cap L'$ then $\varphi_{t} \colon F \to F$ is given by multiplication by a scalar.
Furthermore, if $t,t' \in L \cap L'$, then since the maps of $\varphi$ commute with the structure maps of $\mathbb{I}(L)$ and $\mathbb{I}(L')$, both $\varphi_t$ and $\varphi_{t'}$ are given by multiplication by the same scalar $x \in F$.
This scalar $x$ can only be nonzero if $b \leq b'$ and $d \leq d'$.
This can be shown by examining the components of a morphism in the case that either $b'<b$ or $d' < d$: if $b'<b$, then choosing any $t \in \interval{b',b}$, for any $t' \in L \cap L'$ we have $t < t'$ and thus
\[
\varphi_{t'}(v) = \varphi_{t'}(\mathbb{I}(L')_{t \leq t'}(v)) = \mathbb{I}(L)_{t \leq t'}(\varphi_t(v)) = \mathbb{I}(L)_{t \leq t'}(0) = 0.
\]
A similar argument applies if $d' < d$.
Finally, if $b \leq b' < d \leq d'$, then any choice of $x$ yields a morphism, as the commutativity conditions are satisfied for $t,t' \in L \cap L'$ and are trivial otherwise; and in this case, distinct choices of $x$ yield distinct morphisms since $L \cap L' \neq \varnothing$.
We have thus observed the following, previously noted in~\cite[Example~4.1]{bubenik2021homological} and \cite[Equation~5]{dey2024efficient}, for instance.

\begin{lemma}\label{lemma:morphisms_of_interval_modules}
    Given $L = \interval{b,d}$ and $L'=\interval{b',d'}$,
    \[
    \Hom \big( \mathbb{I}(L'), \mathbb{I}(L) \big) \cong \begin{cases}
        F & \text{ if $b \leq b' < d \leq d'$}\\
        0 & \text{ otherwise,}\\
    \end{cases}
    \]
    where the isomorphism sends $\varphi \colon \mathbb{I}(L') \to \mathbb{I}(L)$ to the $x \in F$ such that $\varphi_t$ is given by multiplication by $x$ for any $t \in L \cap L'$.
\end{lemma}

This understanding of interval modules lets us describe morphisms of direct sums.
Let $L_i = \interval{b_i, d_i}$ and $L'_j = \interval{b'_j, d'_j}$ for all $i \in I$ and $j \in J$, with $I$ and $J$ finite.
If $U = \bigoplus_{j \in J} \mathbb{I}(L'_j)$ and $V = \bigoplus_{i \in I} \mathbb{I}(L_i)$, then
\begin{equation}\label{equation:first description of hom sets}
\Hom (U,V) \cong \bigoplus_{\substack{i \in I \\ j \in J}} \Hom \big( \mathbb{I}(L'_j), \mathbb{I}(L_i) \big).
\end{equation}
This isomorphism sends $\varphi \colon U \to V$ to $\big\{ \varphi^{i,j} \big\}_{i,j}$, where each $\varphi^{i,j}$ is the composition of $\varphi$ with inclusion and projection maps: $\mathbb{I}(L'_j) \to U \xrightarrow[]{\varphi} V \to \mathbb{I}(L_i)$.
By Lemma~\ref{lemma:morphisms_of_interval_modules}, identifying each $\varphi^{i,j}$ with a scalar $x_{i,j}$ gives
\begin{equation}\label{equation:second description of hom sets}
\Hom (U,V) \cong \Big\{ \{x_{i,j}\}_{\substack{i \in I \\ j \in J}} \,\Bigr\vert\, \text{$x_{i,j} \in F$, $x_{i,j}=0$ if $b_i > b'_j$ or $d_i > d'_j$ or $b'_j \geq d_i$} \Big\},
\end{equation}
where now the morphism $\varphi \colon U \to V$ is sent to the matrix $X = \{x_{i,j}\}_{i,j}$ such that the composition $\mathbb{I}(L'_j)_t \to U_t \xrightarrow[]{\varphi_t} V_t \to \mathbb{I}(L_i)_t$ is given by multiplication by $x_{i,j}$ for any $t \in L_i \cap L'_j$.

It is possible to describe composition of morphisms using the matrices above, but it turns out that the operation is simplified by first obtaining an alternate description of the hom-set, derived from presentations of interval-decomposable modules.
First, we note that an interval module $\mathbb{I}\interval{b,d}$ can be presented by the short exact sequence
\[
\begin{tikzcd}
0 \arrow[r] & {\mathbb{I} \interval{d, \infty}} \arrow[r, hook] & {\mathbb{I} \interval{b, \infty}} \arrow[r, two heads] & {\mathbb{I} \interval{b,d}} \arrow[r] & 0.
\end{tikzcd}
\]
If $d = \infty$, then $\mathbb{I} \interval{\infty, \infty}$ is the zero module. 
Later we will also see cases where it is convenient to allow $b=d$, although for a first reading of this section, and throughout the work on the isometry theorem, it is safe to restrict to cases where $b < d$.

We can take finite direct sums of these presentations: given a persistence module $V \cong \bigoplus_{i \in I} \mathbb{I}\interval{b_i, d_i}$ of finite type, we have a short exact sequence 
\[
\begin{tikzcd}
0 \arrow[r] & {\bigoplus_{i \in I} \mathbb{I} \interval{d_i, \infty}} \arrow[r, hook] & {\bigoplus_{i \in I} \mathbb{I} \interval{b_i, \infty}} \arrow[r, two heads] & V \arrow[r] & 0.
\end{tikzcd}
\]
Elements of $V_t$ can thus be represented by equivalence classes of (column) vectors $v = \{v_i\}_{i \in I}$.
Each entry is constrained to be zero before the corresponding bar is born, that is, $v_i = 0$ if $t \in \interval{-\infty, b_i}$; and each entry is quotiented out after the corresponding bar dies, that is, $\{v_i\}_i$ is equivalent to $\{v'_i\}_i$ in $V_t$ if and only if $v_i = v'_i$ for all $i$ such that $t \in \interval{-\infty, d_i}$.
We will denote the equivalence class of $v$ by $[v]$ when the module $V$, its interval decomposition, and the parameter $t$ are understood.
This naturally leads to matrices with quotiented entries, as in the following theorem.

\begin{theorem}\label{theorem: hom set isomorphism}

    Given persistence modules of finite type $U \cong \bigoplus_{j \in J} \mathbb{I}\interval{b'_j,d'_j}$ and $V \cong \bigoplus_{i \in I} \mathbb{I}\interval{b_i,d_i}$, we have an isomorphism
    \[
    \Hom(U,V) \cong \Big\{ 
    \{x_{i,j}\}_{\substack{i \in I \\ j \in J}} \,\Bigr\vert\, x_{i,j} = 0 \text{ if $b_i>b'_j$ or $d_i>d'_j$} \Big\} \Bigr/ \sim,
    \]
    where $\{x_{i,j}\}_{i,j} \sim  \{x'_{i,j}\}_{i,j}$ if $x_{i,j} = x'_{i,j}$ whenever $b'_j<d_i$.
    Under this isomorphism, $\varphi \colon U \to V$ corresponds to the equivalence class of a matrix $X = \{x_{i,j}\}_{i,j}$ such that the composition of $\varphi_t$ with inclusion and projection maps $\mathbb{I}\interval{b'_j,d'_j}_t \to U_t \xrightarrow[]{\varphi_t} V_t \to \mathbb{I}\interval{b_i,d_i}_t$ is given by multiplication by $x_{i,j}$ for any $t \in \interval{b'_j,d'_j} \cap \interval{b_i,d_i}$.
    Furthermore, composition of morphisms corresponds to matrix multiplication: under the respective isomorphisms, if $\varphi \colon U \to V$ and $\omega \colon W \to U$ correspond to the classes of matrices $X$ and $Y$, then $\varphi \circ \omega$ corresponds to the class of $XY$.
\end{theorem}

\begin{proof}

We use presentations of $U$ and $V$ as described above. 
Any pair of morphisms $\psi$ and $\chi$ making the left square in the diagram below commute define a morphism $\varphi \colon U \to V$.
\[
\begin{tikzcd}
0 \arrow[r] & {\bigoplus_{j \in J} \mathbb{I} \interval{d'_j, \infty}} \arrow[r, hook] \arrow[d, "\psi"'] & {\bigoplus_{j \in J} \mathbb{I} \interval{b'_j, \infty}} \arrow[r, two heads] \arrow[d, "\chi"'] & U \arrow[r] \arrow[d, "\varphi"] & 0 \\
0 \arrow[r] & {\bigoplus_{i \in I} \mathbb{I} \interval{d_i, \infty}} \arrow[r, hook]                     & {\bigoplus_{i \in I} \mathbb{I} \interval{b_i, \infty}} \arrow[r, two heads]                     & V \arrow[r]                      & 0
\end{tikzcd}
\]
Such a pair of maps $\psi$ and $\chi$ can be defined from any matrix $X = \{x_{i,j}\}_{i,j}$ in
\[
M = \Big\{ 
    \{x_{i,j}\}_{\substack{i \in I \\ j \in J}} \,\Bigr\vert\, x_{i,j} = 0 \text{ if $b_i>b'_j$ or $d_i>d'_j$} \Big\}.
\]
That is, $\chi$ is defined by requiring that for each $i$ and $j$ and for all $t \in \interval{b'_j, \infty} \cap \interval{b_i, \infty}$, the map ${\mathbb{I} \interval{b'_j, \infty}_t \to \mathbb{I} \interval{b_i, \infty}_t}$ obtained by composing inclusion and projection maps with $\chi_t$ is given by multiplication by $x_{i,j}$, and $\psi$ is defined analogously.
This condition is satisfied vacuously by the zero map in cases of infinite birth or death times, and otherwise, Lemma~\ref{lemma:morphisms_of_interval_modules} assures these maps on interval modules are well defined.
It follows that components of $\varphi$ are computed as stated in the theorem.

The linear map $f \colon M \to \Hom(U,V)$ sending the matrix $X$ to the corresponding $\varphi$ is surjective, since $X$ can be chosen to be the matrix identified with a given morphism by Equation~\ref{equation:second description of hom sets}.
The claimed isomorphism $\Hom(U,V) \cong M / \sim$ will thus follow if we show $\ker f$ consists of all $X \in M$ such that $x_{i,j} = 0$ if $b'_j<d_i$.
Any $X$ of this form is in $\ker f$, since then $\im \chi$ is contained in $\bigoplus_{i \in I} \mathbb{I} \interval{d_i, \infty}$.
For the converse, suppose $X \in \ker f$ and $b'_j<d_i$.
If $b_i>b'_j$ or $d_i>d'_j$, then the definition of $M$ shows $x_{i,j} = 0$.
If not, then $b_i \leq b'_j < d_i \leq d'_j$, so there exists a $t \in \interval{b'_j,d'_j} \cap \interval{b_i,d_i}$, and since the composition $\mathbb{I}\interval{b'_j,d'_j}_t \to U_t \xrightarrow[]{f(X)_t} V_t \to \mathbb{I}\interval{b_i,d_i}_t$ is given by multiplication by $x_{i,j}$, we have $x_{i,j} = 0$, as required.

The fact that composition of morphisms corresponds to matrix multiplication now follows by composing morphisms of short exact sequences like the above, and we check that the product is a matrix of the correct form as follows.
Suppose we have birth times $\{b''_k\}_{k \in K}$ and a matrix $Y = \{y_{j,k}\}_{j \in J, k \in K}$ such that $y_{j,k} = 0$ if $b'_j > b''_k$.
If $b_i > b''_k$, then for all $j$ we have $b_i > b'_j$ or $b'_j > b''_k$ and thus $x_{i,j} = 0$ or $y_{j,k} = 0$, so the $i,k$ entry of $XY$ is $\sum_{j} x_{i,j} y_{j,k} = 0$.
The analogous argument applies for death times.
\end{proof}

We will mostly disregard the isomorphism of Equation~\ref{equation:second description of hom sets} and instead use the identification of morphisms with classes of matrices given by Theorem~\ref{theorem: hom set isomorphism} throughout\footnote{Retroactively, we can determine the operation on the matrices of Equation~\ref{equation:second description of hom sets} corresponding to composition of morphisms.
If the class of $X$ represents $\varphi$ under the isomorphism of Theorem~\ref{theorem: hom set isomorphism}, then replacing all entries $x_{i,j}$ such that $b'_j \geq d_i$ with $0$ yields the matrix corresponding to $\varphi$ in Equation~\ref{equation:second description of hom sets}.
Thus, with the isomorphisms of Equation~\ref{equation:second description of hom sets}, composition of morphisms corresponds to an altered matrix multiplication, in which we first multiply matrices, then change the required entries in the resulting matrix to zero.
The more convenient description of composition in Theorem~\ref{theorem: hom set isomorphism} is the reason we will prefer the isomorphism of the theorem.
However, we will briefly see the other perspective in Section~\ref{subsection: infinite barcodes} when generalizing to modules with infinitely many bars.
}. 
Once interval modules have been appropriately indexed and elements identified with classes of vectors, Theorem~\ref{theorem: hom set isomorphism} gives the expected formula for computing morphisms: if $X$ represents a morphism $\varphi \colon U \to V$ in the chosen interval decompositions and $[u] \in U_t$, then 
\begin{equation}\label{equation:computation_of_morphism}
    \varphi_t([u]) = [Xu].
\end{equation}
The following definitions will provide a useful language for working with these matrices.
Figure~\ref{fig: matrix representation of a morphism} provides a visualization in terms of barcodes.

\begin{definition}
    In the notation of the theorem, when the hom-set and indexing of interval modules is understood, we will denote the equivalence class of a matrix $X$ by $[X]$, and we will say that $X$ \emph{represents} the morphism $\varphi$.
    For compatible hom-sets, we multiply classes of matrices by setting $[X][Y] = [XY]$.
    The $i,j$ entries of a matrix such that $d_i \leq b'_j$ play a special role, as these are the ones that can be altered without affecting the equivalence class: we will call these \emph{trivial entries}, and all other entries will be called nontrivial.
\end{definition}
\begin{figure}[h]
    \centering
    \includegraphics[width=.96\linewidth]{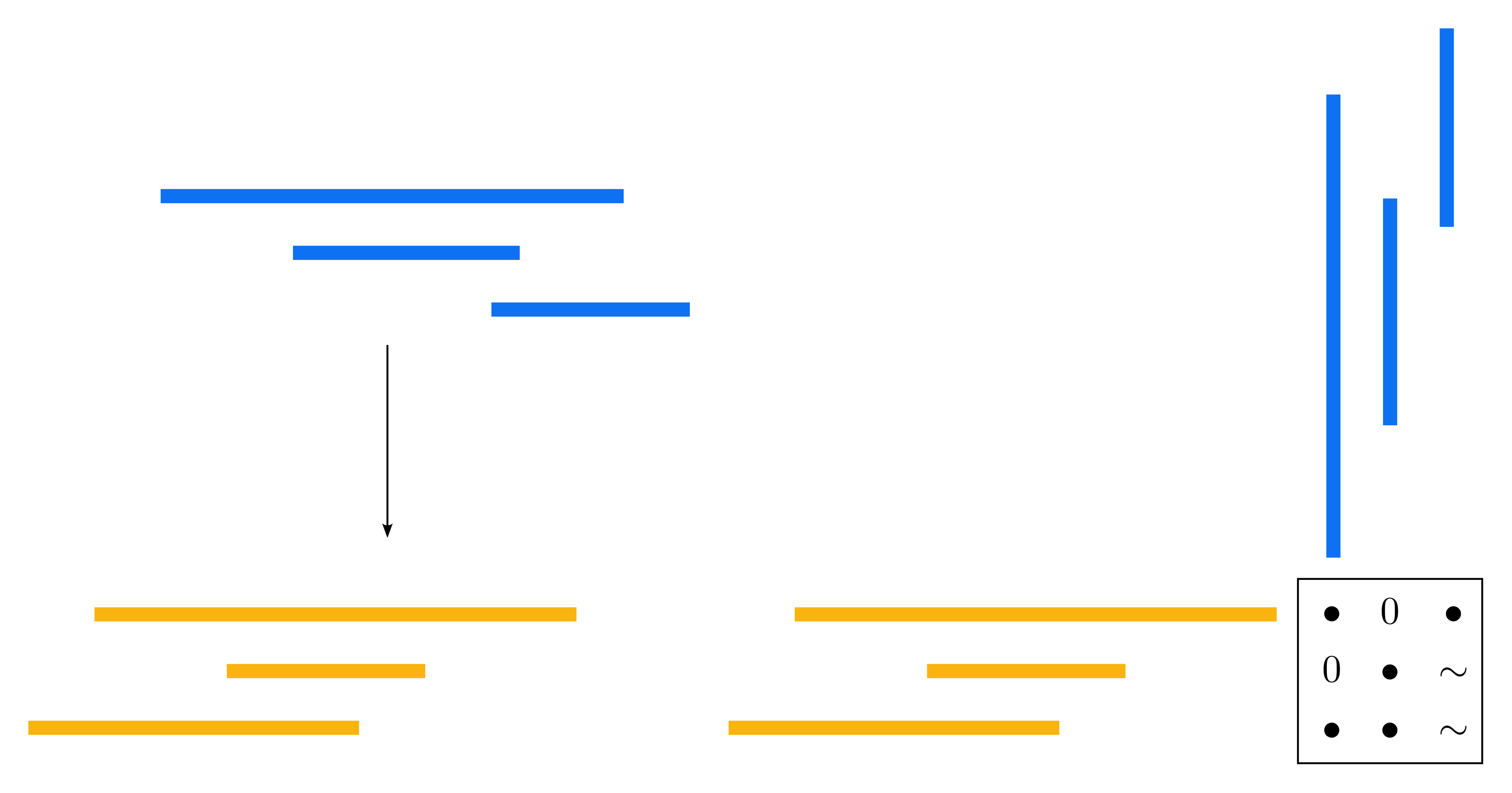}
    \caption{Visualization of a matrix representation of a morphism.
    The relative positions of bars, as show on the left, determine the form of the matrix.
    On the right, the bars of the domain are rotated to align with the corresponding matrix entries.
    In the matrix, the zeros indicate entries that are constrained to be zero, each $\bullet$ indicates a nontrivial entry that can take any value, and the entries marked with $\sim$ are trivial.}
    \label{fig: matrix representation of a morphism}
\end{figure}

\section{Matrix Operations}\label{section:matrix operations for persistence modules}

Here we record some basic operations on matrices representing morphisms of persistence modules.
Since we have seen that the relative positions of birth and death times are important to understanding morphisms, orderings of bars will be useful.
We make the following definition, which specifies orders with a certain symmetry between bars of the domain and codomain.
The benefit of this choice of orders will be evidenced in upcoming proofs.

\begin{definition}
    Let $L_i = \interval{b_i, d_i}$ and $L'_j = \interval{b'_j, d'_j}$ for each $i \in I$ and $j \in J$, and let $U = \bigoplus_{j \in J} \mathbb{I}(L'_j)$ and $V = \bigoplus_{i \in I} \mathbb{I}(L_i)$.
    We will say that total orders $<_J$ and $<_I$ on $J$ and $I$ are \emph{standard orders} for the ordered pair $(U,V)$ if they satisfy $j <_J j'$ if $b'_j < b'_{j'}$ and $i <_I i'$ if $d_{i} > d_{i'}$.
    That is, standard orders for $(U,V)$ order $J$ by birth times of intervals and order $I$ in reverse by death times, with ties broken arbitrarily.
\end{definition}

We will usually omit the subscripts and write all orders using the usual symbols $<$ and $\leq$.
We will also say a matrix representing a morphism $U \to V$ is in standard order if the index sets for its columns and rows have been given standard orders for the pair $(U,V)$.
In a matrix in standard order, trivial entries are clustered in the lower right corner: if $d_i \leq b'_j$, then for any $i' \geq i$ and $j' \geq j$, we have $d_{i'} \leq b'_{j'}$.
See Figure~\ref{fig: matrix representation of a morphism} for an example.

\subsection{Row and column operations and matrix reduction}

For matrices in standard order representing morphisms of persistence modules, we will see that row/column operations are useful when lesser rows/columns are added to greater rows/columns.
We can perform matrix reduction using only these operations in the obvious way: to be explicit, we make the following definition.

\begin{definition}\label{definition: ordered matrix reduction}
    We will call the following operation \emph{ordered column reduction}.
    Given a matrix with ordered rows and columns, perform the following steps, beginning with the least column.
    \begin{enumerate}
        \item\label{item:column reduction step 1} If all entries of the current column are zero, make the next column the current column and repeat this step.
        \item Select as a current pivot the first nonzero entry in the current column.
        \item Add multiples of the current column to each greater than it so as to make the entry in the same row as the current pivot zero.
        \item Make the next column the current column and return to step~\ref{item:column reduction step 1}.
    \end{enumerate}
    Define \emph{ordered row reduction} analogously, interchanging the roles of rows and columns.\footnote{A more general version of ordered row/column reduction could allow multiplying pivot rows/columns by nonzero scalars, for instance, to make the pivots $1$'s.
    This operation can readily be included in computations, but since we will not have any theoretical use for it, we use this more restrictive definition of reduction that does not include a scaling operation.}
\end{definition}

If $X$ is a matrix that represents a morphism $U \to V$ in standard order, ordered column reduction of $X$ produces a matrix $X' = XT$, where $T$ is a square, upper triangular matrix with $1$'s on the diagonal ($T$ is the product of elementary matrices corresponding to the column operations, each found by applying the operation to the identity matrix).
Since rows and columns are not permuted during the reduction, $X'$ is not necessarily in the typical ``staircase'' form obtained when reducing with permutations, but a pivot is still recognized as the first nonzero element in a column, and entries in the same row as a pivot and in a greater column are zero. 
We will continue to use the language of trivial and nontrivial entries for $X'$ (and in fact, we will see in Section~\ref{section: computations} that $X'$ does represent a certain morphism).
In particular, any pivot that is a trivial entry of the matrix will be called a \emph{trivial pivot}.
In a column with a trivial pivot, all entries before the pivot are zero and all entries after are trivial entries, so the column represents a zero element in the codomain.
The following lemma shows that pivot positions are the same for ordered row or column reduction.

\begin{lemma}[Invariance of pivot positions]\label{lemma:same_pivots}
    Ordered row reduction of a matrix produces pivots in the same positions as ordered column reduction.
\end{lemma}

\begin{proof}
    Ordered row and column operations do not change the ranks of upper-left submatrices, that is, submatrices of $X = \{x_{i,j}\}_{i \in I,\, j \in J}$ of the form $\{x_{i,j}\}_{i \leq i_0,\, j \leq j_0}$, so their ranks remain unchanged throughout ordered row or column reduction.
    In either reduced matrix, the ranks of all upper-left submatrices completely determine the pivot positions, so ordered row reduction or column reduction of $X$ must produce pivots in the same positions.
\end{proof}

The proof in fact implies that pivot positions can be considered an invariant of an ordered matrix under ordered row and column operations.
Ranks of submatrices can be used in a similar way in the context of the persistent homology algorithm: see the section ``Pairing'' in~\cite[VII.1]{edelsbrunner_and_harer}.

\subsection{Transposes and Dual Persistence Modules}\label{subsection: transposes and dual modules}

As with vector spaces, if a matrix represents a morphism of persistence modules, then the transpose represents a certain dual morphism.
Let $R^\mathrm{op}$ be the index set $R$ with the order reversed.
Given a persistence module $V$ indexed by $R$, we get a dual module $V^\ast$ indexed by $R^\mathrm{op}$ by letting $V^\ast_t$ be the dual space of $V_t$ and letting each structure map $V^\ast_t \to V^\ast_{t'}$ be the dual of the map $V_{t'} \to V_t$.
Given a morphism $\varphi \colon U \to V$, we get a dual morphism $\varphi^\ast \colon V^\ast \to U^\ast$ by letting each $\varphi^\ast_t$ be the dual of $\varphi_t$.
We have the expected natural isomorphism of pointwise finite-dimensional persistence modules $V^{\ast \ast} \cong V$, with components given by the natural isomorphisms $(V_t)^{\ast \ast} \to V_t$.

Here we write $\mathbb{I}_{R}(L)$ and $\mathbb{I}_{R^\mathrm{op}}(L)$ for the interval modules on $L$ indexed by $R$ and $R^\mathrm{op}$ respectively.
If $V \cong \bigoplus_{i \in I} \mathbb{I}_R(L_i)$ is pointwise finite dimensional, then $V^\ast \cong \bigoplus_{i \in I} \mathbb{I}_{R^\mathrm{op}}(L_i)$:
to construct a specific interval decomposition of $V^\ast$, begin with the basis of each $V_t$ specified by the interval decomposition of $V$ and use the dual basis for $V^\ast_t$, noting that this agrees with the structure maps.
Thus, $V$ and $V^\ast$ have the same barcode, although the structure maps of the interval modules they represent are reversed.
We also have the following representation of dual morphisms.

\begin{lemma}\label{lemma: dual modules and transposes}
    Let $L_i = \interval{b_i, d_i}$ and $L'_j = \interval{b'_j, d'_j}$ for each $i \in I$ and $j \in J$, with $I$ and $J$ finite, and let $U \cong \bigoplus_{j \in J} \mathbb{I}_R(L'_j)$ and $V \cong \bigoplus_{i \in I} \mathbb{I}_R(L_i)$.
    Suppose $X = \{x_{i,j}\}_{i,j}$ represents a morphism $\varphi \colon U \to V$.
    Then using dual bases to define decompositions $U^\ast \cong \bigoplus_{j \in J} \mathbb{I}_{R^\mathrm{op}}(L'_j)$ and $V^\ast \cong \bigoplus_{i \in I} \mathbb{I}_{R^\mathrm{op}}(L_i)$, the transpose $X^{\mathrm{T}}$ represents the dual map $\varphi^\ast \colon V^\ast \to U^\ast$.
\end{lemma}

\begin{proof}
    We split matrix entries into three cases. 
    First, if $L_i \cap L'_j \neq \varnothing$, then we proceed as usual with linear maps: it can be checked that $x_{i,j}$ is the $j,i$ entry of the matrix representing $\varphi^\ast$ by computing $\varphi^\ast_t$ for any $t \in L_i \cap L'_j$ in terms of the dual bases.
    For the second case, if $x_{i,j}$ is a trivial entry of $X$, then $L_i$ dies before $L'_j$ is born in the order of $R$.
    Reversing the order, this means $L'_j$ dies before $L_i$ is born in the order of $R^\mathrm{op}$, so the $j,i$ entry of the matrix representing $\varphi^\ast$ is indeed trivial as well.
    Finally, we use a similar argument if $L_i \cap L'_j = \varnothing$ and $x_{i,j}$ is nontrivial, i.e., $b'_j<d_i$.
    In this case, we must have $b_i > b'_j$ or $d_i > d'_j$, since otherwise any $t \in \interval{b'_j, d_i}$ would be in both $L'_j$ and $L_i$.
    This implies $x_{i,j}=0$ by the form of the matrices in Theorem~\ref{theorem: hom set isomorphism}, and similarly, reversing the order shows the $j,i$ entry of the matrix representing $\varphi^\ast$ must be zero as well.
\end{proof}

In the proof, we saw a convenient property of dual maps expressed in terms of dual bases: the $i,j$ entry of $X$ is trivial if and only if the $j,i$ entry of $X^{\mathrm{T}}$ is trivial.
The indexing of dual bases above also respects standard orders: $(<_J, <_I)$ is a standard order for $(U,V)$ if and only if $(<_I, <_J)$ is a standard order for $(V^\ast, U^\ast)$.
Thus, arguments depending on standard orders also apply to transposes with the same orders.
The representation of a dual morphism by a matrix transpose can be visualized in the right side of Figure~\ref{fig: matrix representation of a morphism} by reflecting both the matrix and the bars across the main diagonal.

\section{The Isometry Theorem for Persistence Modules of Finite Type}\label{section: isometry theorem finite type}

The matrix techniques above are enough to prove the isometry theorem for persistence modules of finite type.
For this section (and the following two), we work with persistence modules indexed by $\mathbb{R}$.
In preparation for the proof, we set the following notation for interleaved persistence modules.
Let $U \cong \bigoplus_{j \in J} \mathbb{I}(L'_j)$ and $V \cong \bigoplus_{i \in I} \mathbb{I}(L_i)$, with $I$ and $J$ finite, and with $L_i = \interval{b_i, d_i}$ and $L'_j = \interval{b'_j, d'_j}$ in each case.
Suppose $\varphi \colon U \to V_{\_ + \varepsilon}$ and $\psi \colon V \to U_{\_ + \varepsilon}$ form an $\varepsilon$-interleaving.
Note that $U_{\_ + \varepsilon} \cong \bigoplus_{j \in J} \mathbb{I}(L'_j-\varepsilon)$ and $V_{\_ + \varepsilon} \cong \bigoplus_{i \in I} \mathbb{I}(L_i-\varepsilon)$, and analogously for $U_{\_ + 2\varepsilon}$ and $V_{\_ + 2\varepsilon}$; we keep this consistent use of the index sets throughout to represent the morphisms by matrices.
Give the index sets $J$ and $I$ standard orders for the pair $(U,V)$.
Let $X = \{x_{i,j}\}$ represent $\varphi$ and let $Y = \{y_{j,i}\}$ represent $\psi$; note that this means $X$ is in standard order but $Y$ is not necessarily.
Since $X$ and $Y$ also represent the shifted morphisms $\varphi_{\_ + \varepsilon} \colon U_{\_ + \varepsilon} \to V_{\_ + 2\varepsilon}$ and $\psi_{\_ + \varepsilon} \colon V_{\_ + \varepsilon} \to U_{\_ + 2\varepsilon}$, the product $YX$ represents the shift map $U \to U_{\_ + 2\varepsilon}$ and $XY$ represents the shift map $V \to V_{\_ + 2\varepsilon}$.
Performing ordered column reduction of $X$ as in the previous section, we can write the reduced matrix as $\{x'_{i,j}\} = X' = XT$, where $T$ has rows and columns indexed by $J$ and is upper triangular with 1's on the diagonal.
For any $j \in J$, let $e_j$ be the vector indexed by $J$ with a $j$ entry of $1$ and zeros for all other entries.
Finally, the last step of the proof will use the dual modules $U^\ast$ and $V^\ast$, which are indexed by $\mathbb{R}^{\mathrm{op}}$; we will follow~\cite{Bauer_Lesnick_2015} and identify $\mathbb{R}^{\mathrm{op}}$ with $\mathbb{R}$ by sending $t$ to $-t$.
With this identification, the dual maps $\varphi^\ast$ and $\psi^\ast$ form an $\varepsilon$-interleaving of $U^\ast$ and $V^\ast$.

\begin{figure}[ht]
    \centering
    \includegraphics[width=0.509\linewidth]{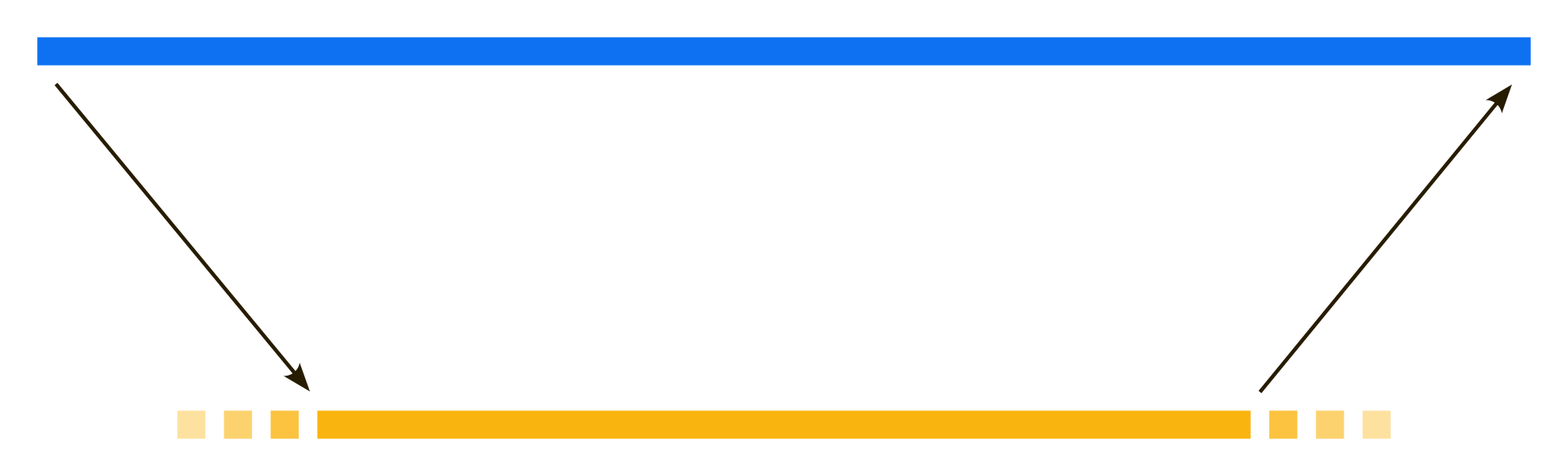}
    \caption{In an interleaving, each sufficiently long bar must survive as it is mapped into the other module and back, suggesting there must be a bar in the other module at least long enough to sustain it.
    }
    \label{fig: interleaving}
\end{figure}

The intuition for algebraic stability is shown in Figure~\ref{fig: interleaving} on a single pair of bars.
Matrix reduction will be the key to showing that all long enough bars can be matched simultaneously.

\begin{theorem}[Algebraic stability for persistence modules of finite type]\label{theorem: algebraic stability finite modules}
    Let $U$ and $V$ be persistence modules of finite type indexed by $\mathbb{R}$, and suppose $X$ represents one morphism $U \to V_{\_ + \varepsilon}$ of an $\varepsilon$-interleaving in standard order.
    Then after ordered column (or row) reduction of $X$, the positions of the nontrivial pivots form an $\varepsilon$-matching between $\barcode(U)$ and $\barcode(V)$.
\end{theorem}

\begin{proof}
    In the notation above, let $M \subseteq I \times J$ be the matching consisting of all pairs $(i,j)$ such that the $i,j$ entry of $X'$ is a nontrivial pivot.
    We first check that any $L'_{j}$ containing a closed interval of length $2\varepsilon$, say $[t,t+2\varepsilon]$, is matched in $M$.
    The vector $Te_{j}$ represents an element of $U_t$, since by the ordering of $J$, all nonzero entries correspond to bars born before $t$.
    Since $Te_{j}$ has a nontrivial $j$ entry of $1$, the shift map sends it to an element of $U_{t+2 \varepsilon}$ that also has a nontrivial $j$ entry of $1$. 
    In particular, since $U_{t \leq t+2\varepsilon} = \psi_{t+\varepsilon} \circ \varphi_t$, we know $\varphi_t([Te_{j}]) = [XTe_{j}] = [X'e_{j}]$ must be nonzero, so the $j$ column of $X'$ must contain a nontrivial pivot.
    Therefore $L'_{j}$ is matched in $M$.

    Next, we show that for any $(i,j) \in M$, we have $\interval{b'_{j}, d'_{j}} \subseteq \interval{b_{i} - \varepsilon, d_{i} + \varepsilon}$.
    For any $(i,j) \in M$, the entry $x'_{i,j}$ is nonzero and nontrivial, so
    $b_i - \varepsilon \leq b'_j < d_i - \varepsilon \leq d'_j$;
    in particular, we get the first required bound $b_{i} - \varepsilon \leq b'_{j}$.
    If $L'_{j}$ does not contain a closed interval of length $2 \varepsilon$, then $d'_{j} \leq b'_{j} + 2 \varepsilon < d_{i} + \varepsilon$, which establishes the required bound on death times.
    On the other hand, if $L'_{j}$ does contain a closed interval $[t, t+ 2 \varepsilon]$, we again use the fact that the $j$ component of $U_{t \leq t+2\varepsilon}([Te_{j}]) = [YX'e_{j}]$ must be $1$, that is, $\sum_{i' \in I} y_{j,i'} x'_{i',j} = 1$.
    Since $x'_{i,j}$ is a pivot, $x'_{i',j} = 0$ for all $i'<i$.
    Thus, there must be some $i' \geq i$ such that $y_{j,i'} \neq 0$, and since $Y$ represents the map $V \to U_{\_+\varepsilon}$, this implies $d'_{j} - \varepsilon \leq d_{i'}$.
    Since $i' \geq i$, we further have $d_{i'} \leq d_{i}$, so $d'_{j} \leq d_{i} + \varepsilon$ as required.

    Finally, we dualize using Lemma~\ref{lemma: dual modules and transposes} to check the remaining conditions for $M$ to be an $\varepsilon$-matching. 
    The transposed matrices $X^{\mathrm{T}}$ and $Y^{\mathrm{T}}$ represent the morphisms of an $\varepsilon$-interleaving of the dual modules $U^\ast$ and $V^\ast$.
    Ordered column reduction of $X^{\mathrm{T}}$ corresponds to ordered row reduction $X$, so by Lemma~\ref{lemma:same_pivots}, the pivot positions obtained by column reducing $X^{\mathrm{T}}$ are the $j,i$ positions such that $(i,j) \in M$.
    Applying the results above with $X^{\mathrm{T}}$ in place of $X$ shows that each $L_i$ containing a closed interval of length $2\varepsilon$ is matched in $M$ and each pair $(i,j) \in M$ satisfies $\interval{b_{i}, d_{i}} \subseteq \interval{b'_{j} - \varepsilon, d'_{j} + \varepsilon}$.
\end{proof}

The use of column reduction in the proof is suspiciously similar to its use in the persistent homology algorithm; we will look at the reason for this similarity in Section~\ref{subsection: comparison with algebraic stability}.
We finish the remaining converse direction of the isometry theorem for persistence modules of finite type in the proof below.
The method is well known (see~\cite{Lesnick_2015_interleaving, Bauer_Lesnick_2015,chazal2016structure}, for instance), and we present it here in the language of matrix algebra.

\begin{theorem}[Isometry theorem for persistence modules of finite type]\label{theorem:isometry theorem for finite modules}
    Let $U$ and $V$ be persistence modules of finite type indexed by $\mathbb{R}$.  
    There exists an $\varepsilon$-interleaving between $U$ and $V$ if and only if there exists an $\varepsilon$-matching between $\barcode(U)$ and $\barcode(V)$.  
    This implies
    \[
    d_I(U,V) = d_B(\barcode(U), \barcode(V)).
    \]
\end{theorem}

\begin{proof}
    Half of the proof is accomplished by Theorem~\ref{theorem: algebraic stability finite modules}, which constructs a matching using the pivots of a matrix.
    For the converse, we construct an interleaving from the simplest matrices with chosen pivots. 
    Given an $\varepsilon$-matching between $\{L_i\}_{i \in I}$ and $\{L'_j\}_{j \in J}$, define $Z = \{z_{i,j}\}_{i \in I, j \in J}$ by setting $z_{i,j} = 1$ if $L_i$ is matched to $L'_j$ and letting all other entries be zero.
    Then $Z$ represents a morphism $U \to V_{\_ + \varepsilon}$, since the $\varepsilon$-matching conditions assure $b_i - \varepsilon \leq b'_j$ and $d_i - \varepsilon \leq d'_j$ if $L_i$ and $L'_j$ are matched.
    Similarly, $Z^{\mathrm{T}}$ represents a morphism $V \to U_{\_ + \varepsilon}$.
    The composition $U \to V_{\_ + \varepsilon} \to U_{\_ + 2\varepsilon}$ is represented by $Z^{\mathrm{T}} Z$, which has a $j,j$ entry of $1$ for each matched $L'_j$ and has zeros for all other entries.
    Since each $L'_j$ that contains a closed interval of length $2 \varepsilon$ is matched, $Z^{\mathrm{T}} Z$ represents the shift map $U \to U_{\_ + 2 \varepsilon}$.
    Similarly, $Z Z^{\mathrm{T}}$ represents the shift map $V \to V_{\_ + 2 \varepsilon}$, so the maps represented by $Z$ and $Z^{\mathrm{T}}$ define an $\varepsilon$-interleaving.
\end{proof}

It is worth comparing the proof of Theorem~\ref{theorem: algebraic stability finite modules} to the closely related methods of~\cite{bjerkevik_stability} and~\cite{Bauer_Lesnick_2015}.
Our use of $YX$ to represent the shift map is similar to the approach of~\cite{bjerkevik_stability} (which uses multiparameter persistence modules).
The proof of~\cite[Lemma~4.9]{bjerkevik_stability} uses the rank of the matrix representing the shift map to bound counts of matchable bars.
The proof of Theorem~\ref{theorem: algebraic stability finite modules} above can thus be understood as realizing this rank-based approach in terms of pivots.
The connection with~\cite{Bauer_Lesnick_2015} will be clearer after Section~\ref{section: computations}, as we show in Theorem~\ref{theorem: barcodes of images} that column reduction of a matrix can be used to compute the image of a morphism.
In particular, the factorization of a morphism through its image in~\cite[Section~5]{Bauer_Lesnick_2015} is replaced by column reduction of a matrix in Theorem~\ref{theorem: algebraic stability finite modules}. 
See Section~\ref{subsection: comparison with algebraic stability} for further discussion.
Similarly to the induced matchings of~\cite{Bauer_Lesnick_2015}, Theorem~\ref{theorem: algebraic stability finite modules} constructs a matching from just one morphism of an interleaving.
However, the two methods are not equivalent, as we now check.

\begin{example}
    The following simple example shows that the matchings given by Theorem~\ref{theorem: algebraic stability finite modules} can differ from those constructed in~\cite{Bauer_Lesnick_2015}.
    Let $U = \mathbb{I}(L'_1) \oplus \mathbb{I}(L'_2)$ and $V = \mathbb{I}(L_1) \oplus \mathbb{I}(L_2)$ with
    \begin{align*}
        \hspace{4cm} L'_1 &= [0,5) & L_1 &= [1,5)\hspace{4cm}\\
        \hspace{4cm}L'_2 &= [0,4) & L_2 &= [1,4). \hspace{4cm}
    \end{align*}
    We have $d_I(U,V) = d_B(\barcode(U), \barcode(V)) = 1$, and a pair of maps $\varphi \colon U \to V_{\_+1}$ and $\psi \colon V \to U_{\_+1}$ form a $1$-interleaving if and only if they are represented by a pair of inverse matrices.
    Supposing they form an interleaving, we must have $\im \varphi = V$.
    The matching induced by $\varphi$ as defined in~\cite{Bauer_Lesnick_2015} only depends on this image: it will always match $L'_1$ to $L_1$ and $L'_2$ to $L_2$ (see in particular~\cite[Proposition~5.4]{Bauer_Lesnick_2015}).
    However, a matching constructed from $\varphi$ using Theorem~\ref{theorem: algebraic stability finite modules} depends on the matrix $X = \{x_{i,j}\}$ representing $\varphi$.
    If $x_{1,1} = 0$, then ordered column reduction produces pivots in the off-diagonal positions, matching $L'_1$ to $L_2$ and $L'_2$ to $L_1$.
    If $x_{1,1} \neq 0$, then the pivots are in the diagonal positions, matching $L'_1$ to $L_1$ and $L'_2$ to $L_2$.
\end{example}

The isometry theorem is typically stated for q-tame modules, a more general setting than we have considered so far.
Sections~\ref{section: isometry theorem locally finite} and~\ref{section: isometry theorem q-tame} provide a proof of the isometry theorem for q-tame modules beginning with the results of this section for persistence modules of finite type. These sections do not use matrix algebra or the method used for Theorem~\ref{theorem: algebraic stability finite modules}, but instead, the proofs proceed by comparing more general persistence modules to modules of finite type.
Readers who would prefer to continue to further matrix algebra and persistent homology computations can safely skip these sections, as they are not used in the rest of the paper.

\section{The Isometry Theorem for Locally Finite Modules}\label{section: isometry theorem locally finite}

While the main purpose of this paper has been to give the matrix-based proof in the previous section, this section and the next show that more general versions of the isometry theorem can be proved using modules of finite type as a starting point.
I previously wrote about the methods used here in~\cite{moy2024persistence}.
This section generalizes to locally finite modules, and the next section generalizes to q-tame modules.

We will use the following operation that restricts a persistence module's support: given a persistence module $V$ indexed by $\mathbb{R}$ and an interval $\mathcal{I} \subseteq \mathbb{R}$, we define $V|_\mathcal{I}$ by letting $(V|_\mathcal{I})_t = V_t$ if $t \in \mathcal{I}$ and letting $(V|_\mathcal{I})_t = 0$ otherwise.
For $s \leq t$, we let $(V|_\mathcal{I})_{s \leq t} = V_{s \leq t}$ if $s,t \in \mathcal{I}$, and let $(V|_\mathcal{I})_{s \leq t}$ be the zero map otherwise.
Restricting the support respects direct sums, as they are defined pointwise.

The intuition for extending the algebraic stability theorem to locally finite persistence modules is straightforward.
Restricting the support of a pair of interleaved locally finite modules to any finite interval will give interleaved modules of finite type, so we can apply our existing version of algebraic stability, Theorem~\ref{theorem: algebraic stability finite modules}.
We should then be able to piece together matchings on a sequence of growing intervals to obtain a matching for the original pair of modules.
The proof below follows this outline, with some of the verifications left to the reader.

\begin{theorem}[The Isometry Theorem for Locally Finite Persistence Modules]\label{theorem:isometry_theorem_for_locally_finite}
For locally finite persistence modules $U$ and $V$ indexed by $\mathbb{R}$, there exists an $\varepsilon$-interleaving between $U$ and $V$ if and only if there exists an $\varepsilon$-matching between $\barcode(U)$ and $\barcode(V)$.
This implies
    \[
    d_I(U,V) = d_B(\barcode(U), \barcode(V)).
    \]
\end{theorem}

\begin{proof}
Suppose $U = \bigoplus_{k \in K} \mathbb{I}(L'_{k})$ and $V = \bigoplus_{j \in J} \mathbb{I}(L_j)$ are locally finite persistence modules, and fixing $\varepsilon>0$, suppose $\varphi \colon U \to V_{\_+\varepsilon}$ and $\psi \colon V \to U_{\_ + \varepsilon}$ define an $\varepsilon$-interleaving.
Then for any interval $\mathcal{I} \subseteq \mathbb{R}$, we get an $\varepsilon$-interleaving of $U|_\mathcal{I}$ and $V|_\mathcal{I}$ in which the $t$ components are $\varphi_t$ and $\psi_t$ whenever $t$ and $t+\varepsilon$ are both in $\mathcal{I}$ and otherwise are zero maps.
Letting
\[
J(\mathcal{I}) = \{ j \in J \mid L_j \cap \mathcal{I} \neq \varnothing\}
\]
\[
K(\mathcal{I}) = \{ k \in K \mid L'_{k} \cap \mathcal{I} \neq \varnothing\},
\]
we have interval decompositions $U|_\mathcal{I} \cong \bigoplus_{k \in K(\mathcal{I})} \mathbb{I}(L'_{k} \cap \mathcal{I})$ and $V|_\mathcal{I} \cong \bigoplus_{j \in J(\mathcal{I})} \mathbb{I}(L_j \cap \mathcal{I})$.
Let $\mathcal{M}(\mathcal{I})$ be the set of $\varepsilon$-matchings $M \subseteq J(\mathcal{I}) \times K(\mathcal{I})$ between $\dgm(V|_\mathcal{I})$ and $\dgm(U|_\mathcal{I})$\footnote{For vacuous cases, note that the empty set may be an $\varepsilon$-matching and can be in $\mathcal{M}(\mathcal{I})$ even if $J(\mathcal{I})$ or $K(\mathcal{I})$ is empty.}.
For any bounded $\mathcal{I}$, since $U$ and $V$ are locally finite, $U|_\mathcal{I}$ and $V|_\mathcal{I}$ are of finite type, so Theorem~\ref{theorem: algebraic stability finite modules} shows $\mathcal{M}(\mathcal{I})$ is nonempty.

We now use matchings on a growing sequence of intervals to find a matching between $\dgm(V)$ and $\dgm(U)$.
For intervals $\mathcal{I} \subseteq \mathcal{I}'$, we get a function $\mathcal{M}(\mathcal{I}' \supseteq \mathcal{I}) \colon \mathcal{M}(\mathcal{I}') \to \mathcal{M}(\mathcal{I})$ defined by 
\[
\mathcal{M}(\mathcal{I}' \supseteq \mathcal{I})(M) = M \cap (J(\mathcal{I}) \times K(\mathcal{I})).
\]
Let $\mathcal{I}_1 = [-1,1]$.
If $\mathcal{I}_1 \subseteq \mathcal{I} \subseteq \mathcal{I}'$, then $\im \mathcal{M}(\mathcal{I}' \supseteq \mathcal{I}_1) \subseteq \im \mathcal{M}(\mathcal{I} \supseteq \mathcal{I}_1)$, that is, the image can only shrink as the interval grows.
Since $\mathcal{M}(\mathcal{I}_1)$ is finite, the image can only shrink finitely many times, so there exists a bounded $\mathcal{I}_2$ containing $\mathcal{I}_1$ such that for any bounded $\mathcal{I}$ containing $\mathcal{I}_2$, $\im \mathcal{M}(\mathcal{I} \supseteq \mathcal{I}_1) = \im \mathcal{M}(\mathcal{I}_2 \supseteq \mathcal{I}_1)$.
Repeat to recursively define a sequence of bounded intervals $\mathcal{I}_n$ such that for each $n$, $\mathcal{I}_{n} \subseteq \mathcal{I}_{n+1}$ and for any $\mathcal{I}$ containing $\mathcal{I}_{n+1}$, $\im \mathcal{M}(\mathcal{I} \supseteq \mathcal{I}_n) = \im \mathcal{M}(\mathcal{I}_{n+1} \supseteq \mathcal{I}_n)$.
We may further require $[-n,n] \subseteq \mathcal{I}_n$ for each $n$, since after each step, we can expand the interval as needed.
We have seen above that all $\mathcal{M}(\mathcal{I}_n)$ are nonempty, so choose $M_1 \in \im \mathcal{M}(\mathcal{I}_2 \supseteq \mathcal{I}_1)$.
Since $\im \mathcal{M}(\mathcal{I}_{n+2} \supseteq \mathcal{I}_n) = \im \mathcal{M}(\mathcal{I}_{n+1} \supseteq \mathcal{I}_n)$ for each $n$, we can recursively choose $M_{n+1} \in \im \mathcal{M}(\mathcal{I}_{n+2} \supseteq \mathcal{I}_{n+1})$ such that $\mathcal{M}(\mathcal{I}_{n+1} \supseteq \mathcal{I}_n)(M_{n+1}) = M_n$.
This gives a sequence of matchings $M_1 \subseteq M_2 \subseteq M_3 \dots$ that account for increasingly large intervals, and it can be checked that $\bigcup_n M_n$ is an $\varepsilon$-matching between $\dgm(V)$ and $\dgm(U)$.

We have thus shown that if locally finite modules $U$ and $V$ are $\varepsilon$-interleaved, then there is an $\varepsilon$-matching between $\dgm(U)$ and $\dgm(V)$: this is algebraic stability for locally finite modules.
For the converse, the well-known method used in the proof of Theorem~\ref{theorem:isometry theorem for finite modules} still applies, although the finite matrices there are no longer sufficient to describe the morphisms\footnote{Section~\ref{subsection: infinite barcodes} defines infinite matrices that can be used to give an equivalent proof.}.
Instead, given an $\varepsilon$-matching $M \subseteq J \times K$, we define a morphism $\varphi^M \colon U \to V_{\_+\varepsilon}$ by letting the composition of $\varphi^M_t$ with inclusion and projection maps
\[
\mathbb{I}(L'_k)_{t} \to U_t \xrightarrow[]{\varphi^M_t} V_{t + \varepsilon} \to \mathbb{I}(L_j)_{t + \varepsilon}
\]
be the identity when $(j,k) \in M$, $t \in L'_k$, and $t + \varepsilon \in L_j$ and letting it be the zero map otherwise.
Defining $\psi^M \colon V \to U_{\_+\varepsilon}$ analogously, $\varphi^M$ and $\psi^M$ define an $\varepsilon$-interleaving since the interleaving conditions are satisfied on interval module summands.
\end{proof}

\section{The Isometry Theorem for Q-tame modules}\label{section: isometry theorem q-tame}

Generalizing the isometry theorem to q-tame modules indexed by $\mathbb{R}$ requires an appropriate definition of persistence diagrams for q-tame modules and some basic properties.
Those familiar with~\cite{chazal2016structure} can use the definitions there and find the proofs of the facts in the lemma below.
We give an equivalent definition that will allow us to quickly prove the necessary results.

Following a key theme in~\cite[Section~4.5]{chazal2016structure}, we approximate q-tame modules by locally finite modules.
Namely, given a q-tame module $V$ and an $\varepsilon>0$, the \emph{$\varepsilon$-smoothing} $V^\varepsilon$ is defined by $V^\varepsilon_t = \im V_{t-\varepsilon \leq t+ \varepsilon}$, with structure maps the restrictions of those of $V$.
Q-tameness of $V$ implies $V^\varepsilon$ is pointwise finite dimensional (in fact, it is locally finite, as we will show below), so $V^\varepsilon$ is interval decomposable (Theorem~\ref{theorem: pfd implies interval decomposable}) and thus has well-defined barcodes and persistence diagrams.
For an interval-decomposable $V \cong \bigoplus_{i \in I} \mathbb{I}\interval{b_i,d_i}$, we have $V^\varepsilon \cong \bigoplus_{i \in I^\varepsilon} \mathbb{I}\interval{b_i+\varepsilon,d_i-\varepsilon}$, where $I^\varepsilon$ consists of those $i \in I$ such that $\interval{b_i,d_i}$ contains a closed interval of length $2 \varepsilon$.
To counteract this shrinking of intervals, we define a shifted diagram of an interval-decomposable module:
let $\overline{b_i}, \overline{d_i} \in \mathbb{R}\cup\{\pm \infty\}$ be obtained by removing decorations from $b_i$ and $d_i$ and define
\[
\dgm^\varepsilon \left( \bigoplus_{i \in I} \mathbb{I}\interval{b_i,d_i} \right) = \left\{ \left( \overline{b_i} - \varepsilon, \overline{d_i} + \varepsilon\right) \mid i \in I, \overline{b_i} < \overline{d_i}\right\}.
\]
For any q-tame $V$ and any $\varepsilon_2 > \varepsilon_1 > 0$, these shifted diagrams satisfy $\dgm^{\varepsilon_2}(V^{\varepsilon_2}) \subseteq \dgm^{\varepsilon_1}(V^{\varepsilon_1})$, since $V^{\varepsilon_2} = (V^{\varepsilon_1})^{\varepsilon_2 - \varepsilon_1}$.
Furthermore, all points in $\dgm^{\varepsilon_1}(V^{\varepsilon_1})$ with difference between birth and death times greater than $2\varepsilon_2$ are in $\dgm^{\varepsilon_2}(V^{\varepsilon_2})$.
We thus define\footnote{Here we use the evident definition of the union of a collection of multisets, in which the multiplicity of an element of the union is its maximum multiplicity in any multiset in the collection (here guaranteed to be finite).
Using indexed multisets, this union can be indexed by the colimit of the index sets of the shifted diagrams.
} $\dgm(V) = \bigcup_{\varepsilon>0} \dgm^{\varepsilon}(V^{\varepsilon})$.
This agrees with the definition given in~\cite{chazal2016structure} by~\cite[Proposition~4.16]{chazal2016structure}.
Using this definition, we record some facts about q-tame modules and smoothings in the following lemma.

\begin{lemma}\label{lemma: properties of smoothings}
    Let $V$ be a q-tame persistence module indexed by $\mathbb{R}$ and let $\varepsilon>0$.
    Then the following hold.
    \begin{enumerate}
        \item\label{item: smoothings are locally finite} $V^\varepsilon$ is locally finite
        \item\label{item: bound on d_I for smoothings} $d_I(V,V^\varepsilon) \leq \varepsilon$
        \item\label{item: bound on d_B for smoothings}  $d_B(\dgm(V), \dgm(V^\varepsilon)) \leq \varepsilon$
    \end{enumerate}
\end{lemma}

\begin{proof}
    For (\ref{item: smoothings are locally finite}), we have $V^\varepsilon = (V^{\frac{\varepsilon}{2}})^{\frac{\varepsilon}{2}}$, so an interval in $\barcode(V^\varepsilon)$ intersects $[t - \frac{\varepsilon}{2}, t+ \frac{\varepsilon}{2}]$ only if the corresponding interval in $\barcode(V^{\frac{\varepsilon}{2}})$ contains $t$.
    Since $V$ is q-tame, $V^{\frac{\varepsilon}{2}} = \im V_{t -\frac{\varepsilon}{2} \leq t+\frac{\varepsilon}{2}}$ is finite dimensional, so we conclude that only finitely many intervals in $\barcode(V^\varepsilon)$ intersect $[t - \frac{\varepsilon}{2}, t+ \frac{\varepsilon}{2}]$.
    For (\ref{item: bound on d_I for smoothings}), the collection of maps $V_{t -\varepsilon \leq t+ \varepsilon} \colon V_{t - \varepsilon} \to V^\varepsilon_t$ and the collection of inclusions $V^\varepsilon_t \hookrightarrow V_{t + \varepsilon}$ define the two morphisms of an $\varepsilon$-interleaving.
    To show (\ref{item: bound on d_B for smoothings}), define a matching by the function that sends each $(\overline{b},\overline{d})$ in $\dgm(V^\varepsilon)$ to the corresponding $(\overline{b}-\varepsilon,\overline{d}+\varepsilon)$ in $\dgm(V)$; any unmatched point in $\dgm(V)$ does not have a difference between birth and death greater than $2 \varepsilon$, so this is in fact an $\varepsilon$-matching.
\end{proof}

Since we have already established the isometry theorem for locally finite modules, the isometry theorem for q-tame modules can now be proved by using smoothings as approximations. 

\begin{theorem}[Isometry theorem for q-tame persistence modules]\label{theorem: isometry theorem q-tame}
For any q-tame persistence modules $U$ and $V$ indexed by $\mathbb{R}$,
\[
d_I(U,V) = d_B(\dgm(U), \dgm(V)).
\]
\end{theorem}

\begin{proof}
For any $\varepsilon>0$, Lemma~\ref{lemma: properties of smoothings}(\ref{item: smoothings are locally finite}), shows $U^\varepsilon$ and $V^\varepsilon$ are locally finite, so by Theorem~\ref{theorem:isometry_theorem_for_locally_finite}, we have $d_I(U^\varepsilon, V^\varepsilon) = d_B(\dgm(U^\varepsilon), \dgm(V^\varepsilon))$. 
Applying the triangle inequality for $d_I$, Lemma~\ref{lemma: properties of smoothings}(\ref{item: bound on d_I for smoothings}) implies that $d_I(U,V)$ and $d_I(U^\varepsilon, V^\varepsilon)$ differ by at most $2 \varepsilon$.
Similarly, Lemma~\ref{lemma: properties of smoothings}(\ref{item: bound on d_B for smoothings}) implies $d_B(\dgm(U), \dgm(V))$ and $d_B(\dgm(U^\varepsilon), \dgm(V^\varepsilon))$ differ by at most $2 \varepsilon$.
Thus, $d_I(U,V)$ and $d_B(\dgm(U), \dgm(V))$ differ by at most $4 \varepsilon$, and since this holds for all $\varepsilon>0$, we have $d_I(U,V) = d_B(\dgm(U), \dgm(V))$.
\end{proof}

\section{A Category of Barcodes}\label{section: a category of barcodes}

This section and the next develop matrix algebra for persistence modules in more detail, and the ideas are applied to the persistent homology algorithm in Section~\ref{section: functorial PH algorithm}.
Here we present a category of finite barcodes that is equivalent to the category of persistence modules of finite type.
The morphisms will be the equivalence classes of matrices described in Theorem~\ref{theorem: hom set isomorphism}, and the equivalence of categories essentially follows from the theorem.
This is the categorical formalization of the statement that matrix algebra allows for computations with persistence modules, and it closely parallels the classic equivalence of categories between the category of finite vector spaces and the category of matrices (\cite[Exercise~I.4.6]{mac1971categories}, \cite[Corollary~1.5.11]{riehl2016category}).
It is also closely related to equivalences between categories of persistence modules and categories of graded modules (see \cite[Theorem~3.1]{zomorodian_carlsson_computing} and \cite[Remark~2.1]{carlsson-filippenko_PH_sum_metric}, for instance).

We also take this as an opportunity to explicitly discuss ``empty bars,'' which have equal birth and death time $b=d$.
The interval module for such a bar is the zero module, presented by the short exact sequence
\[
\begin{tikzcd}
0 \arrow[r] & {\mathbb{I} \interval{b, \infty}} \arrow[r, hook] & {\mathbb{I} \interval{b, \infty}} \arrow[r, two heads] & {\mathbb{I} \interval{b,b}} \arrow[r] & 0.
\end{tikzcd}
\]
Including zero modules presented as such in interval decompositions will be helpful in the following sections, where they arise naturally -- they can simply be removed when they are no longer helpful.
Theorem~\ref{theorem: hom set isomorphism} and the resulting matrix computations apply with such terms included: note that we never assumed that birth times are strictly less than death times in the proof.
To include an empty bar at $b$ in the decomposition of $U$, we let a vector representing an element of $U_t$ have an entry corresponding to this bar, which is constrained to be zero if $t \in \interval{-\infty, b}$ and is a trivial entry if $t \in \interval{b,\infty}$.
Corresponding entries in matrices will similarly either be zero or trivial.

What is lost when using empty bars is the ability to associate a unique birth and death time to a nonempty interval.
We will still want to distinguish between different birth times $b_1 \neq b_2$ in spite of the identical sets $\interval{b_1, b_1} = \varnothing = \interval{b_2, b_2}$ and identical interval modules $\mathbb{I}\interval{b_1, b_1} = 0 = \mathbb{I}\interval{b_2, b_2}$.
We thus represent bars by ordered pairs\footnote{Arguably, the terminology of persistence diagrams would be more fitting here. 
Our use of ordered pairs is most similar to the decorated persistence diagrams of~\cite{chazal2016structure} for persistence modules indexed by $\mathbb{R}$, although these don't allow pairs of equal decorated reals and are thus equivalent to barcodes without empty bars.
While the construction of the objects is not important from a categorical perspective, an alternate definition that would stay more true to barcodes could use intervals as nonempty bars and cuts of $R$ as empty bars.
See also the category barcodes without empty bars, $\overline{\cB}$, after Definition~\ref{definition:category of barcodes}.
} $(b,d)$ with $b \leq d$.
Using this language, we define a category of barcodes that allow empty bars.

\begin{definition}\label{definition:category of barcodes}
    Define a category $\cB_R$ of finite barcodes with endpoints in $R$ as follows.
    Let objects be finite indexed multisets $\{(b_i,d_i)\}_{i \in I}$ of ordered pairs of cuts of $R$ with $b_i \leq d_i$.
    Define the hom-sets by
    \[
    \cB_R \Big(\{(b'_j, d'_j)\}_{j \in J}, \{(b_i,d_i)\}_{i \in I}\Big) = \Big\{ 
    \{x_{i,j}\}_{\substack{i \in I \\ j \in J}} \,\Bigr\vert\, x_{i,j} = 0 \text{ if $b_i>b'_j$ or $d_i>d'_j$} \Big\} \Bigr/ \sim,
    \]
    where we identify $\{x_{i,j}\} \sim \{x'_{i,j}\}$ if $x_{i,j} = x'_{i,j}$ whenever $b'_j < d_i$. 
    A morphism is thus an equivalence class of a matrix $X=\{x_{i,j}\}$, written as $[X]$ when the hom-set is understood.
    Given morphisms 
    \[
    \{(b''_k, d''_k)\}_{k \in K} \xrightarrow[]{[Y]} \{(b'_j, d'_j)\}_{j \in J} \xrightarrow[]{[X]} \{(b_i,d_i)\}_{i \in I},
    \]
    define composition by matrix multiplication: $[X][Y] = [XY]$. 
\end{definition}

We mostly omit the subscript $R$, writing the category of barcodes as $\cB = \cB_R$.
Composition is well defined by Theorem~\ref{theorem: hom set isomorphism}, or it can also be checked on the classes of matrices directly by showing that matrix multiplication respects the equivalence classes.
Adjusting the definition to require $b_i < d_i$ produces a category $\overline{\cB}$ of barcodes without empty bars (in which intervals can then be used instead of ordered pairs of endpoints).
The functor $\cB \to \overline{\cB}$ that removes empty bars is an equivalence of categories, with inverse given by the inclusion $\overline{\cB} \to \cB$.
We will use $\cB$ below, but readers can check that the rest of this section could be carried out analogously using $\overline{\cB}$ instead of $\cB$.

We now relate $\cB$ to the category of persistence modules of finite type indexed by $R$, which we denote $\cP = \cP_R$.
First, any barcode can be realized by a direct sum of interval modules: by a slight abuse of notation, we will extend the use of the symbol $\mathbb{I}$ for interval modules and let a functor $\mathbb{I} \colon \cB \to \cP$ be defined on objects by $\mathbb{I}(\{(b_i,d_i)\}_{i \in I}) = \bigoplus_{i \in I} \mathbb{I}\interval{b_i,d_i}$.
For a morphism $[X] \colon \{(b'_j, d'_j)\}_j \to \{(b_i,d_i)\}_i$, we let $\mathbb{I}([X])$ be the morphism $\bigoplus_{j \in J} \mathbb{I}\interval{b'_j, d'_j} \to \bigoplus_{i \in I} \mathbb{I}\interval{b_i,d_i}$ computed by the matrix $X$, following Equation~\ref{equation:computation_of_morphism}.
That $\mathbb{I}$ is an equivalence of categories can be seen from the fact that it is full, faithful, and essentially surjective on objects~\cite[Theorem~1.5.9]{riehl2016category}, by Theorem~\ref{theorem: hom set isomorphism}.
It can also be checked directly by using choices of interval decompositions to define an inverse, as we do below\footnote{As noted in~\cite[Theorem~1.5.9]{riehl2016category}, the characterization of a functor defining an equivalence of categories as one that is full, faithful, and essentially surjective on objects uses (a categorical version of) the axiom of choice.
This will be made explicit in choices of interval decompositions and their use in Theorem~\ref{theorem: equivalence of categories}.
See also~\cite[Remark~1.5.14]{riehl2016category}.}.

We define a barcode functor $\barcode \colon \cP \to \cB$, which tacitly depends on choices of interval decompositions.
Suppose all persistence modules of finite type are identified with persistence modules of the form $\bigoplus_{i \in I} \mathbb{I}\interval{b_i,d_i}$ by fixed isomorphisms, with $\mathbb{I}(\{(b_i,d_i)\}_{i \in I})$ being identified with itself by the identity morphism in each case.
Given a persistence module of finite type $V$ and the chosen isomorphism $V \cong \bigoplus_{i \in I} \mathbb{I}\interval{b_i,d_i}$, let $\barcode(V) = \{(b_i,d_i)\}_{i \in I}$.
If $\varphi \colon U \to V$ is a morphism of persistence modules of finite type with $\barcode(U) = \{(b'_j, d'_j)\}_{j \in J}$ and $\barcode(V) = \{(b_i,d_i)\}_{i \in I}$, let $\widetilde{\varphi}$ be the composition with the chosen isomorphisms:
\[
    \bigoplus_{j \in J} \mathbb{I}\interval{b'_j, d'_j} \xrightarrow[]{\cong} U \xrightarrow[]{\varphi} V \xrightarrow[]{\cong} \bigoplus_{i \in I} \mathbb{I}\interval{b_i,d_i}.
\]
Define $\barcode(\varphi) \colon \barcode(U) \to \barcode(V)$ to be the class of a matrix $X$ such that $\widetilde{\varphi}$ is computed by multiplication by $X$, as in Equation~\ref{equation:computation_of_morphism}.
The class is uniquely determined, by Theorem~\ref{theorem: hom set isomorphism} (and we can in fact always take $X$ to be the matrix identified with $\widetilde{\varphi}$ in Equation~\ref{equation:second description of hom sets}).

\begin{theorem}\label{theorem: equivalence of categories}
    The functors $\barcode$ (defined with any choice of interval decompositions as above) and $\mathbb{I}$ form an equivalence of categories between $\cP$ and $\cB$.
\end{theorem}

\begin{proof}
    Since $\barcode \circ \mathbb{I}$ is the identity on $\cB$, we just need to check that $\mathbb{I} \circ \barcode$ is naturally isomorphic to the identity functor on $\cP$.
    The components of the natural isomorphisms are those chosen in the definition of $\barcode$: given a morphism of persistence modules $\varphi \colon U \to V$ with $\barcode(U) = \{(b'_j, d'_j)\}_{j \in J}$ and $\barcode(V) = \{(b_i,d_i)\}_{i \in I}$, the following diagram commutes by the definition of $\barcode(\varphi)$.
    \[
    \begin{tikzcd}[/tikz/baseline=(tikz@f@1-2-1.base)]
    U \arrow[r, "\cong"] \arrow[d, "\varphi"'] & {\bigoplus_{j \in J} \mathbb{I}\interval{b'_j, d'_j}} \arrow[d, "\mathbb{I} \circ \barcode(\varphi)"] \\
    V \arrow[r, "\cong"']                      & {\bigoplus_{i \in I} \mathbb{I}\interval{b_i, d_i}}                                                  
    \end{tikzcd}
    \qedhere
    \]
\end{proof}

In this categorical setting, once choices of interval decompositions have been made, dual modules and matrix transposes define contravariant functors as shown below, where subscripts indicate the index set.
By Lemma~\ref{lemma: dual modules and transposes}, the square formed by $\barcode_R$ and $\barcode_{R^\mathrm{op}}$ commutes if the interval decompositions defining $\barcode_{R^\mathrm{op}}$ use the dual bases of the interval decompositions that define $\barcode_R$, and the square formed by $\mathbb{I}_R$ and $\mathbb{I}_{R^\mathrm{op}}$ commutes up to natural isomorphism.
\[
\begin{tikzcd}
\mathcal{P}_R \arrow[r, "\mathrm{bar}_R", shift left] \arrow[d, "\ast"']          & \mathcal{B}_R \arrow[l, "\mathbb{I}_R", shift left] \arrow[d, "^\mathrm{T}"]        \\
\mathcal{P}_{R^\mathrm{op}} \arrow[r, "\mathrm{bar}_{R^\mathrm{op}}", shift left] & \mathcal{B}_{R^{\mathrm{op}}} \arrow[l, "\mathbb{I}_{R^{\mathrm{op}}}", shift left]
\end{tikzcd}
\]

The functor $\barcode$ forms an additional functorial step at the end of the persistent homology pipeline.
Earlier steps in the computation of persistent homology are frequently functorial, including various methods to transform a space or filtration of spaces into a filtered chain complex and the definition of a persistent homology module from this chain complex.
Thus, in many cases, the entire process from space to barcode can be made functorial, where maps between spaces (or filtrations) are mapped to classes of matrices between barcodes.
In Section~\ref{section: functorial PH algorithm}, we will describe how to compute morphisms between persistent homology barcodes with matrix operations, showing morphisms could in fact be incorporated into persistent homology algorithms.

Finally, the category of barcodes defined here also provides an alternate perspective to the nonexistence result in~\cite[Proposition~5.10]{Bauer_Lesnick_2015}, which states that there cannot exist a functor from the category of (in their case, pointwise finite-dimensional) persistence modules to a category of matchings that sends each persistence module to its barcode.
By replacing the category of matchings there with our category of barcodes, we are able to construct a functor sending each persistence module to its barcode.
Theorem~\ref{theorem: algebraic stability finite modules} then shows how morphisms of barcodes can be used to find matchings, even though the mapping of morphisms to matchings cannot be functorial.

\subsection{Generalizing to infinite barcodes}\label{subsection: infinite barcodes}

Here we briefly observe that matrix representations of morphisms can be adjusted to allow for infinite barcodes, leaving details to the interested reader.
Presentations of interval modules are less convenient in this setting, so instead we adapt the isomorphism of Equation~\ref{equation:second description of hom sets} to the infinite case.

In general, if $U \cong \bigoplus_{j \in J} \mathbb{I}\interval{b'_j, d'_j}$ and $V \cong \bigoplus_{i \in I} \mathbb{I}\interval{b_i,d_i}$, then $\Hom(U,V)$ is isomorphic to the space of (possibly infinite) matrices $\{x_{i,j}\}_{i,j}$ such that 1) $x_{i,j}$ may be nonzero only if $b_i \leq b'_j < d_i \leq d'_j$ and 2) for all $j$ and all $t \in \interval{b'_j, d'_j}$, there are finitely many $i$ such that $t \in \interval{b_i, d_i}$ and $x_{i,j} \neq 0$.
To meet the second condition, it is sufficient, but not always necessary, to require that for each $j$, only finitely many $x_{i,j}$ are nonzero.
For an example in which it is not necessary, using persistence modules indexed by $\mathbb{R}$, consider morphisms $\mathbb{I}(0,1) \to \bigoplus_{n \geq 1} \mathbb{I}(0, \frac{1}{n})$.

For morphisms $\bigoplus_{k \in K} \mathbb{I}\interval{b''_k, d''_k} \to \bigoplus_{j \in J} \mathbb{I}\interval{b'_j, d'_j} \to \bigoplus_{i \in I} \mathbb{I}\interval{b_i,d_i}$,
where the first is represented by $\{y_{j,k}\}_{j,k}$ and the second by $\{x_{i,j}\}_{i,j}$, the composition is represented by $\{z_{i,k}\}_{i,k}$ where
\[
z_{i,k} = \begin{cases}
    \sum_{j \in J} x_{i,j}y_{j,k} & \text{if $b_i \leq b''_k < d_i \leq d''_k$} \\
    0 & \text{otherwise}.
\end{cases}
\]
It can be checked that the sum is always well-defined, as it has only finitely many nonzero terms.
Like in the finite case, this yields a category of barcodes equivalent to the category of interval-decomposable modules.
It is also possible to restrict to full subcategories of interval-decomposable modules meeting additional conditions.
In particular, if the index set $R$ meets the conditions in Theorem~\ref{theorem: pfd implies interval decomposable}, there is a category of pointwise finite-dimensional persistence modules and an equivalent category of pointwise finite barcodes, i.e., those such that only finitely many bars contain any given $t \in R$.

\section{Computations of images, kernels, and cokernels}\label{section: computations}

Here we return to give a more thorough treatment of row and column operations on matrices representing morphisms of persistence modules of finite type, using them to compute images, kernels, and cokernels.
In Section~\ref{section:matrix operations for persistence modules}, we described how to perform ordered column reduction of a matrix $X$ in standard order representing a morphism $U \to V$. 
This produces a reduced matrix $XT$, where $T$ is upper triangular with 1's on the diagonal.
Even though $T$ is invertible as a matrix, it does not necessarily represent an automorphism of $U$, as shown in the following example.
This is a key difference from matrix representations of linear maps.

\begin{example}
    With index set $R = \mathbb{R}$, let $U = \mathbb{I}(L'_1) \oplus \mathbb{I}(L'_2)$ and $V = \mathbb{I}(L_1) \oplus \mathbb{I}(L_2)$ with
    \begin{align*}
        \hspace{4cm} L'_1 &= [0,4) & L_1 &= [0,2)\hspace{4cm}\\
        \hspace{4cm}L'_2 &= [1,3) & L_2 &= [0,2). \hspace{4cm}
    \end{align*}
    A matrix $X$ representing a morphism $U \to V$ has no trivial entries, and all entries may be nonzero; we can thus consider an $X$ that is not already in column reduced form.
    Performing ordered column reduction and writing the reduced matrix as $XT$ with $T = \{t_{j,k}\}$, we will then have $t_{1,2} \neq 0$.
    However, any matrix representing a morphism $U \to U$ must be diagonal by Theorem~\ref{theorem: hom set isomorphism}, so $T$ does not represent an automorphism of $U$.
    Note that it is only the relative positions of the death times of $L'_1$ and $L'_2$ that have prevented $T$ from representing a morphism $U \to U$.
\end{example}

We set some notation to be used for the rest of this section.
Let $U = \bigoplus_{j \in J} \mathbb{I} \interval{b'_j, d'_j}$ and $V = \bigoplus_{i \in I} \mathbb{I} \interval{b_i, d_i}$ be persistence modules of finite type and give $J$ and $I$ standard orders for the pair $(U,V)$.
Define $\overline{U}$ and $\underline{V}$ by extending the bars of $U$ to $+\infty$ and the bars of $V$ to $-\infty$: that is, $\overline{U} = \bigoplus_{j \in J} \mathbb{I} \interval{b'_j, \infty}$ and $\underline{V} = \bigoplus_{i \in I} \mathbb{I} \interval{-\infty, d_i}$.
We have an epimorphism $\overline{U} \to U$ and a monomorphism $V \to \underline{V}$ represented by the respective identity matrices.
Let a morphism $\varphi \colon U \to V$ be represented by $X$.
Using the collection of pivot positions obtained by either ordered row or column reduction of $X$ (which give the same pivot positions by Lemma~\ref{lemma:same_pivots}), we let $P_I \subseteq I$ be the set of $i \in I$ such that row $i$ has a nontrivial pivot and let $P_J \subseteq J$ be the set of $j \in J$ such that column $j$ has a nontrivial pivot.
We also let $p \colon P_J \to P_I$ send $j$ to the row index of the pivot in column $j$ and let $p' \colon P_I \to P_J$ send $i$ to the column index of the pivot in row $i$.
The computations below will involve reordering the rows and columns of matrices.
This will always be done by changing the order assigned to the index sets rather than permuting elements, that is, $x_{i,j}$ always indicates the same entry regardless of the orders given to the index sets $I$ and $J$.

Any invertible upper triangular matrix $T$ indexed by $J$ represents an automorphism of $\overline{U}$, as the ordering of $J$ by birth times and the infinite death times imply that both $T$ and $T^{-1}$ represent morphisms $\overline{U} \to \overline{U}$.
Dually, any invertible lower triangular matrix $S$ represents an automorphism of $\underline{V}$, and each automorphism can be understood as a change of interval decomposition, analogous to a change of basis.
In particular, if $SXT$ is obtained from $X$ by ordered row and column reduction of $X$ (in either order, though $S$ and $T$ will in general depend on the order), then the map $\overline{U} \to \underline{V}$ represented by $SXT$ has an image, kernel, and cokernel that can be observed directly from the pivot positions.
Below, we relate these back to the morphism $\varphi$ represented by $X$.
We list the following theorems together for reference.

\vspace{1cm}

\begin{theorem}[Barcodes of images]\label{theorem: barcodes of images}
    The persistence module $\im \varphi$ has a barcode consisting of a bar $\interval{b'_j, d_{i}}$ for each nontrivial pivot position $i,j$.
    The factorization of $\varphi$ into a pair of maps, $U \to \im \varphi$, defined by restricting the codomain of $\varphi$, and the inclusion $\im \varphi \to V$, can be represented in two ways\footnote{The two versions of the image can be interpreted as image and coimage.}:  
    \begin{enumerate}
        \item\label{item: image barcode 1}
        Index the bars of $\im \varphi$ by $P_I$: $\im \varphi \cong \bigoplus_{i \in P_I} \mathbb{I}\interval{b'_{p'(i)}, d_i}$.
        Let $SX$ be obtained by ordered row reduction of $X$.
        Then we may represent $U \to \im \varphi$ by the rows of $SX$ indexed by $P_I$ and represent $\im \varphi \to V$ by the columns of $S^{-1}$ indexed by $P_I$.
        \item\label{item: image barcode 2}
        Index the bars of $\im \varphi$ by $P_J$: $\im \varphi \cong \bigoplus_{j \in P_J} \mathbb{I}\interval{b'_j, d_{p(j)}}$.
        Let $XT$ be obtained by ordered column reduction of $X$. 
        Then we may represent $U \to \im \varphi$ by the rows of $T^{-1}$ indexed by $P_J$ and represent $\im \varphi \to V$ by the columns of $XT$ indexed by $P_J$.
    \end{enumerate}
\end{theorem}

\begin{theorem}[Barcodes of kernels]\label{theorem: barcodes of kernels}
    Let $XT$ be obtained by ordered column reduction of $X$ and let
        \[
        b''_j = \begin{cases}
            d_{p(j)} & \text{ if $j \in P_J$} \\
            b'_j & \text{ if $j \in J \setminus P_J$}.
        \end{cases}
        \]
    Reorder the columns of $T$ by $b''_j$ and reorder the rows in reverse by $d'_j$.
    Perform ordered column reduction of $T$ with respect to these orders to obtain $\widetilde{T}$ and let $q(j)$ be the row index of the pivot of $\widetilde{T}$ in column $j$.
    Then $b''_j \leq d'_{q(j)}$ for each $j$, $\ker \varphi \cong \bigoplus_{j \in J} \mathbb{I} \interval{b''_j, d'_{q(j)}}$, and $\widetilde{T}$ represents the inclusion $\ker \varphi \hookrightarrow U$.
\end{theorem}

\begin{theorem}[Barcodes of cokernels]\label{theorem: barcodes of cokernels}
    Let $SX$ be obtained by ordered row reduction of $X$ and let
    \[
    d''_i = \begin{cases}
        b'_{p'(i)} & \text{ if $i\in P_I$} \\
        d_i & \text{ if $i \in I \setminus P_I$}.
    \end{cases}
    \]
    Reorder the rows of $S$ in reverse by $d''_i$ and reorder the columns by $b_i$.
    Perform ordered row reduction of $S$ with respect to these orders to obtain $\widetilde{S}$ and let $q'(i)$ be the column index of the pivot of $\widetilde{S}$ in row $i$.
    Then $b_{q'(i)} \leq d''_i$ for each $i$, $\coker \varphi \cong \bigoplus_{i \in I} \mathbb{I} \interval{b_{q'(i)}, d''_i}$, and $\widetilde{S}$ represents the quotient map $V \to \coker \varphi$.
\end{theorem}

\begin{theorem}[Factoring through kernels and cokernels]\label{theorem: factor through ker and coker}
    In the notation of Theorem~\ref{theorem: barcodes of kernels}, if $W$ is of finite type and $Y$ represents a morphism $\psi \colon W \to U$ with $\varphi \circ \psi = 0$, then $\widetilde{T}^{-1}Y$ represents the map $W \to \ker \varphi$ defined by restricting the codomain of $\psi$.
    Dually, in the notation of Theorem~\ref{theorem: barcodes of cokernels}, if $W'$ is of finite type and $Z$ represents a morphism $\psi' \colon V \to W'$ with $\psi' \circ \varphi = 0$, then $Z \widetilde{S}^{-1}$ represents the morphism $\coker \varphi \to W'$ induced by $\psi'$.
\end{theorem}

Note that Theorems~\ref{theorem: barcodes of kernels} and~\ref{theorem: barcodes of cokernels} allow empty bars in the barcodes of the kernel and cokernel, which we will see arise naturally in the proofs.
The resulting invertible matrices $\widetilde{T}$ and $\widetilde{S}$ are useful for computations with maps that factor through $\ker \varphi$ or $\coker \varphi$, as described in Theorem~\ref{theorem: factor through ker and coker}.
The empty bars can of course be removed when not necessary, in which case the corresponding rows or columns of matrices must be removed as well.

\begin{proof}[Proof of Theorem~\ref{theorem: barcodes of images}]
    For intuition, and for variety, we first give a concrete proof of (\ref{item: image barcode 2}) based on matrix and vector entries, then follow with an essentially dual proof of (\ref{item: image barcode 1}) presented more abstractly. 
    Let $XT$ be obtained by ordered column reduction of $X$. 
    For any $t$ and any vector $u = \{u_j\}_j$ representing an element of $U_t$ in the given interval decomposition, $Tu$ and $T^{-1}u$ also represent elements of $U_t$.
    This follows from the ordering of the bars and the fact that $T$ and $T^{-1}$ are upper triangular: if $u_j$ is nonzero, then $\interval{b'_j, d'_j}$ is born before $t$, as is any $\interval{b'_{j'}, d'_{j'}}$ with $j'<j$, so column $j$ of either $T$ or $T^{-1}$ indeed represents an element of $U_t$.
    Any element of $\im \varphi_t$ can thus be represented by a vector of the form $XTu$ where $u$ represents an element of $U_t$.
    The columns of $XT$ that are born before $t$ (i.e., each column $j$ such that $t \in \interval{b'_j, \infty}$) and represent nonzero elements of $V_t$ in fact represent a basis of $\im \varphi_t$: they are linearly independent because the submatrix of $XT$ consisting of columns born before $t$ and rows that die after $t$ is in column reduced form.
    For $j \in P_J$, column $j$ represents one of these basis elements exactly when $t \in \interval{b'_j, d_{p(j)}}$, since its pivot row has the latest death time among its nonzero rows.
    We thus find that the pivot columns of $XT$ represent a map $\iota \colon \bigoplus_{j \in P_J} \mathbb{I}\interval{b'_j, d_{p(j)}} \to V$ such that each $\iota_t$ is injective and $\im \iota_t = \im \varphi_t$, so we use this as the interval decomposition of $\im \varphi$.
    To represent $\varphi$ with codomain restricted to $\im \varphi$ expressed in this interval decomposition, given any $[u] \in U_t$, we must find the element mapped to $\varphi_t([u]) = [Xu]$ by $\iota$.
    Like above, the entries of $T^{-1}u$ indexed by $P_J$ represent an element of $\im \varphi_t$ in the chosen interval decomposition, and since the entries correspond to the pivot columns of $XT$, its image under $\iota$ is represented by $XTT^{-1}u = Xu$ in $V_t$.
    Therefore, the rows of $T^{-1}$ indexed by $P_J$ represent $\varphi$ with codomain restricted to the image.

    For the proof of (\ref{item: image barcode 1}), we let $SX$ be obtained by ordered row reduction of $X$, then further let $SXT'$ be obtained by ordered column reduction of $SX$.
    Let $\overline{\varphi}$ be the map $\overline{U} \to \underline{V}$ represented by $SXT'$, which has pivots as its only nonzero entries.
    Then $\im \overline{\varphi} \cong \bigoplus_{i \in P_I} \mathbb{I}\interval{b'_{p'(i)}, d_i}$, each interval module $\mathbb{I}\interval{b'_{p'(i)}, d_i}$ being the image of the component $\mathbb{I}\interval{b'_{p'(i)}, \infty}$ of $\overline{U}$.
    The map $V \to \underline{V}$ represented by $S$ is a monomorphism, as it is the composition of the monomorphism $V \to \underline{V}$ represented by the identity matrix and the automorphism of $\underline{V}$ represented by $S$, and similarly, the map $\overline{U} \to U$ represented by $T'$ is an epimorphism.
    This implies $\im \varphi \cong \im \overline{\varphi}$, where the map $V \to \underline{V}$ represented by $S$ restricts to an isomorphism $\im \varphi \to \im \overline{\varphi}$.
    Express the images as monomorphisms, as in the following diagram\footnote{Here and in later diagrams, we write morphisms of persistence modules as matrices representing them.
    To more precisely follow the categorical perspective of Section~\ref{section: a category of barcodes}, the persistence modules can be replaced by their barcodes and the matrices by their equivalence classes, giving a diagram in the category of barcodes.}, with the monomorphism $\im \overline{\varphi} \hookrightarrow \underline{V}$ represented by the columns of the $I \times I$ identity matrix indexed by $P_I$.
    \[
    \begin{tikzcd}
    U \arrow[r, "X"] \arrow[rd, dashed] & V \arrow[r, "S", hook]                                          & \underline{V} \\
                                        & \im \overline{\varphi} \arrow[ru, hook] \arrow[u, dashed, hook] &              
    \end{tikzcd}
    \]
    The map from $\im \overline{\varphi}$ to $V$ represented by the columns of $S^{-1}$ indexed by $P_I$ makes the right triangle commute, so it is a monomorphism onto the image of $\varphi$.
    Representing the map $U \to \im \overline{\varphi}$ by the rows of $SX$ indexed by $P_I$ makes the outermost triangle commute, by the row reduced form of $SX$.
    The commutativity of the right triangle and the outermost triangle and the fact that $S$ represents a monomorphism $V \to \underline{V}$ imply the left triangle commutes, which completes the proof.
\end{proof}

Theorem~\ref{theorem: barcodes of images} shows that for $j \in P_J$, the bar $\interval{b'_j, d'_j}$ of the domain accounts for a bar $\interval{b'_j, d_{p(j)}}$ of the image, which can be thought of as a lower portion of the bar of the domain.
A naive guess for the barcode of the kernel would consist of the remaining upper portions $\interval{d_{p(j)}, d'_j}$ of these bars and the rest of the bars $\interval{b'_j, d'_j}$ for $j \in J \setminus P_J$, but this turns out to be incorrect.
A simple counterexample is the map $\mathbb{I}[0,4) \oplus \mathbb{I}[1,3) \to \mathbb{I}[0,2)$ represented by the matrix $[\hspace{.05cm} 1 \hspace{.2cm} 1 \hspace{.05cm}]$:
the barcode of its kernel is $\{[1,4), [2,3)\}$, with, for instance, generators represented by the vectors 
$[\hspace{.05cm} 1 \hspace{.2cm} -1 \hspace{.05cm}]^{\mathrm{T}}$ and $[\hspace{.05cm} 0 \hspace{.2cm} 1 \hspace{.05cm}]^{\mathrm{T}}$.

However, rank-nullity still holds for each $\varphi_t$, so we know the number of bars of the domain containing $t$ is the sum of the numbers of bars of the kernel and image containing $t$.
From this, it is not surprising that the set of death times of bars of the kernel will be the set of $d'_j$, while the set of birth times will consist of $d_{p(j)}$ for $j \in P_J$ and $b'_j$ for $j \in J \setminus P_J$.
How the birth and death times are paired will depend on $\varphi$, and this is reflected by the second matrix reduction in Theorem~\ref{theorem: barcodes of kernels}.

\begin{proof}[Proof of Theorem~\ref{theorem: barcodes of kernels}]
    Column reduce $X$ in order to produce $XT$, then further row reduce $XT$ in order to produce $S'XT$.
    Then $S'XT$ has pivots as its only nonzero entries, so the kernel of the map $\overline{\varphi} \colon \overline{U} \to \underline{V}$ represented by $S'XT$ can be expressed as a direct sum of the kernels of the compositions ${\mathbb{I}\interval{b'_j, \infty} \hookrightarrow \overline{U} \xrightarrow[]{\overline{\varphi}} \underline{V}}$.
    For $j \in P_J$, this restricts to a nonzero map $\mathbb{I}\interval{b'_j, \infty} \to \mathbb{I}\interval{-\infty, d_{p(j)}}$, which has kernel $\mathbb{I}\interval{d_{p(j)}, \infty}$, and for $j \notin P_J$, the kernel is the entire interval module $\mathbb{I}\interval{b'_j, \infty}$.
    Thus, 
    \[
    \ker \overline{\varphi} \cong \bigoplus_{j \in J} \mathbb{I} \interval{b''_j, \infty},
    \]
    with the birth times $b''_j$ defined as in the theorem, and the inclusion $\ker \overline{\varphi} \hookrightarrow \overline{U}$ is represented by the identity matrix.
    Since the map $V \to \underline{V}$ represented by $S'$ is a monomorphism and the map $\overline{U} \to U$ represented by $T$ is an epimorphism, $\ker \varphi$ is the image of the map $\ker \overline{\varphi} \to U$ represented by $T$, so we apply Theorem~\ref{theorem: barcodes of images}(\ref{item: image barcode 2}).
    We reorder the rows and columns of $T$ to have a standard order for the map $\ker \overline{\varphi} \to U$ and column reduce with respect to this order to define $\widetilde{T}$.
    Since $T$ is an invertible matrix, every column of $\widetilde{T}$ has a pivot (possibly trivial), so let $q(j)$ be the row index of the pivot in column $j$.
    Identifying the nontrivial pivot positions as those $q(j),j$ such that $b''_j<d'_{q(j)}$,   Theorem~\ref{theorem: barcodes of images}(\ref{item: image barcode 2}) shows that
    \begin{equation}\label{equation: first barcode of kernel}
        \ker \varphi \, \cong \bigoplus_{\substack{j \in J \\ b''_j<d'_{q(j)}}} \mathbb{I} \interval{b''_j, d'_{q(j)}}
    \end{equation}
    and that $\ker \varphi \hookrightarrow U$ is represented by the corresponding columns of $\widetilde{T}$.
    The rest of this proof shows that for $j$ satisfying $b''_j \geq d'_{q(j)}$, we in fact have $b''_j = d'_{q(j)}$.
    Because of this, it is natural to include the empty bars in the direct sum and represent the inclusion by the entire matrix $\widetilde{T}$, as stated in the theorem.

    The comparison of $b''_j$ and $d'_{q(j)}$ applies most easily to persistence modules indexed by a subset of $\mathbb{Z}$: the method is to compare total bar lengths in the kernel and the image using rank-nullity of each $\varphi_t$.
    To apply it to the general case of modules indexed by $R$, let $C$ be the set of all birth and death times of the bars of $U$ and $V$ (not a multiset, so an endpoint shared by multiple bars only appears once in $C$).
    Recall that these birth and death times are cuts of $R$.
    Write the elements of $C$ in order as $c_0 < c_1 < \dots < c_N$ and let $\ell \colon C \to \{0,1,\dots, N\}$ be defined by $\ell(c_n) = n$.
    Since the $c_n$ are distinct cuts of $R$, there must exist $t_1, t_2, \dots, t_N \in R$ such that $t_n \in \interval{c_{n-1}, c_{n}}$ for each $n$.
    Let $\mathcal{L}(U) = \sum_{n=1}^{N} \dim U_{t_n}$.
    We see that $\dim U_{t_n}$ is the number of bars of $U$ containing $\interval{c_{n-1}, c_{n}}$, so $\mathcal{L}(U)$ can be thought of as a sum of artificial ``lengths'' of bars, where a bar $\interval{c_n, c_m}$ is assigned a length of $\ell(c_m)-\ell(c_n) = m-n$.
    Similarly, let $\mathcal{L}(\ker \varphi) = \sum_{n=1}^{N} \dim \ker \varphi_{t_n}$ and $\mathcal{L}(\im \varphi) = \sum_{n=1}^{N} \dim \im \varphi_{t_n}$.
    
    Using Theorem~\ref{theorem: barcodes of images}(\ref{item: image barcode 2}) for $\im \varphi$ and the bars of $\ker \varphi$ in Equation~\ref{equation: first barcode of kernel}, summing the contributions of the bars gives
    \[
    \mathcal{L}(U) = 
    \sum_{j \in J} \Big( \ell(d'_j)-\ell(b'_j) \Big),
    \]
    \[
    \mathcal{L}(\ker \varphi)= 
    \sum_{\substack{j \in J \\ b''_j<d'_{q(j)}}} \Big( \ell(d'_{q(j)})-\ell(b''_j) \Big),
        \hspace{1cm}    
    \mathcal{L}(\im \varphi) = 
    \sum_{j \in P_J} \Big( \ell(d_{p(j)})-\ell(b'_j) \Big).
    \]
    On the other hand, applying rank nullity to each $\varphi_{t_n}$ gives
    \begin{align*}
        \mathcal{L}(\ker \varphi) &=
        \mathcal{L}(U) - \mathcal{L}(\im \varphi)\\
        &=\sum_{j \in J} \Big( \ell(d'_j)-\ell(b'_j) \Big) - 
        \sum_{j \in P_J} \Big( \ell(d_{p(j)})-\ell(b'_j) \Big) \\
        &= \sum_{j \in J} \Big( \ell(d'_{j})-\ell(b''_j) \Big)\\
        &= \sum_{j \in J} \Big( \ell(d'_{q(j)})-\ell(b''_j) \Big).
    \end{align*}
    The final sum agrees with the one for the kernel above, except that now all $j \in J$ are included in the sum, so we must have
    \[
    \sum_{\substack{j \in J \\ b''_j \geq d'_{q(j)}}} \Big( \ell(d'_{q(j)})-\ell(b''_j) \Big) = 0.
    \]
    But $\ell(d'_{q(j)})-\ell(b''_j)$ is nonpositive when $b''_j \geq d'_{q(j)}$, so since their sum is zero, each term must satisfy ${\ell(d'_{q(j)})-\ell(b''_j) = 0}$ and thus $b''_j = d'_{q(j)}$.
\end{proof}

\begin{proof}[Proof of Theorem~\ref{theorem: barcodes of cokernels}]
    By duality, the representation of cokernels follows from Theorem~\ref{theorem: barcodes of kernels} and the matrix representations of dual maps in Lemma~\ref{lemma: dual modules and transposes}.
\end{proof}

\begin{proof}[Proof of Theorem~\ref{theorem: factor through ker and coker}]
We prove the first statement of the theorem, and the second follows by duality using Lemma~\ref{lemma: dual modules and transposes}.
We just need to show that $\widetilde{T}^{-1}Y$ represents a morphism $W \to \ker \varphi$ with $\ker \varphi$ represented as in Theorem~\ref{theorem: barcodes of kernels}.
It will then follow that this morphism is $\psi$ with restricted codomain, since its composition with the inclusion $\ker \varphi \hookrightarrow U$ is represented by $\widetilde{T} \widetilde{T}^{-1}Y = Y$.
We let $W \cong \bigoplus_{k \in K} \mathbb{I} \interval{a_k, c_k}$, and we work with presentations of $U$, $W$, and $\ker \varphi$ like in the proof of Theorem~\ref{theorem: hom set isomorphism}.
Since $\widetilde{T}$ represents the inclusion $\ker \varphi \hookrightarrow U$ and $Y$ represents a morphism $W \to U$, we get the following commutative diagram, in which the rows are short exact sequences with morphisms represented by identity matrices.
\[
\begin{tikzcd}
{\bigoplus_{j} \mathbb{I}\interval{d'_{q(j)}, \infty}} \arrow[r, hook] \arrow[d, "\widetilde{T}"'] & {\bigoplus_{j} \mathbb{I}\interval{b''_j, \infty}} \arrow[r, two heads] \arrow[d, "\widetilde{T}"'] & \ker \varphi \arrow[d, "\widetilde{T}"] \\
{\bigoplus_{j} \mathbb{I}\interval{d'_j, \infty}} \arrow[r, hook]                                  & {\bigoplus_{j} \mathbb{I}\interval{b'_j, \infty}} \arrow[r, two heads]                              & U                                       \\
{\bigoplus_{k} \mathbb{I}\interval{c_k, \infty}} \arrow[r, hook] \arrow[u, "Y"]                    & {\bigoplus_{k} \mathbb{I}\interval{a_k, \infty}} \arrow[r, two heads] \arrow[u, "Y"]                & W \arrow[u, "Y"']                      
\end{tikzcd}
\]

To show that $\widetilde{T}^{-1}Y$ represents a morphism $W \to \ker \varphi$, it is enough to show that it represents a morphism from the lower module to the upper module in both the left and middle columns.
In the left column, since $q$ is a permutation, the upper two modules are isomorphic.
Since $\widetilde{T}$ represents a morphism, it must be block upper triangular when the bars of the domain and codomain are ordered by their shared birth times, where rows are in the same block when they have the same birth times and similarly for columns.
Therefore $\widetilde{T}^{-1}$ also has this same block upper triangular form, so it represents a morphism in the reverse direction.
Composing with the morphism represented by $Y$ below it shows that $\widetilde{T}^{-1}Y$ represents a morphism $\bigoplus_{k} \mathbb{I}\interval{c_k, \infty} \to \bigoplus_{j} \mathbb{I}\interval{d'_{q(j)}, \infty}$.

In the center column, the image of $Y$ must be contained in the image of $\widetilde{T}$: this can be checked by a diagram chase using the observation that the image of $\widetilde{T}$ contains $\bigoplus_j \mathbb{I} \interval{d'_j, \infty}$, along with assumption that in the right column, the image of $Y$ is contained in the image of $\widetilde{T}$.
Since $\widetilde{T}$ was defined by column reducing $T$, it can be factored as $\widetilde{T} = TQ$, where $Q$ is upper triangular when ordered by $b''_j$ and $T$ is upper triangular when ordered by $b'_j$.
The upper triangular forms imply $Q$ and $T$ represent automorphisms of the upper and center modules respectively in the center column of the diagram, which let us express the right two columns as follows.
\[
\begin{tikzcd}
{\bigoplus_{j} \mathbb{I}\interval{b''_j, \infty}} \arrow[r, "Q^{-1}", two heads] \arrow[d, hook] & \ker \varphi \arrow[d, "\widetilde{T}"] \\
{\bigoplus_{j} \mathbb{I}\interval{b'_j, \infty}} \arrow[r, "T", two heads]                       & U                                       \\
{\bigoplus_{k} \mathbb{I}\interval{a_k, \infty}} \arrow[r, two heads] \arrow[u, "T^{-1}Y"]        & W \arrow[u, "Y"']                      
\end{tikzcd}
\]
Here the image of the morphism represented by $T^{-1}Y$ is contained in $\bigoplus_{j} \mathbb{I}\interval{b''_j, \infty}$, so $T^{-1}Y$ also represents a morphism $\bigoplus_{k} \mathbb{I}\interval{a_k, \infty} \to \bigoplus_{j} \mathbb{I}\interval{b''_j, \infty}$. 
As $Q^{-1}$ represents an automorphism of $\bigoplus_{j} \mathbb{I}\interval{b''_j, \infty}$, we conclude that $Q^{-1} T^{-1}Y = \widetilde{T}^{-1} Y$ represents a morphism $\bigoplus_{k} \mathbb{I}\interval{a_k, \infty} \to \bigoplus_{j} \mathbb{I}\interval{b''_j, \infty}$, as required.
\end{proof}

\section{Computation of Functorial Persistent Homology Barcodes}\label{section: functorial PH algorithm}

Here we apply the matrix techniques developed so far to the computation of persistent homology.
We begin with a review of the classic persistent homology algorithm in Section~\ref{subsection: classic PH algorithm} and discuss the similarities with our matrix-based proof of algebraic stability (Theorem~\ref{theorem: algebraic stability finite modules}) in Section~\ref{subsection: comparison with algebraic stability}.
We then describe how the classic persistent homology algorithm can be extended to compute induced morphisms of persistent homology in Section~\ref{subsection: induced maps for classic PH algorithm}, showing that functorial persistent homology barcodes, in the sense of Section~\ref{section: a category of barcodes}, are computable.
In Section~\ref{subsection: computation with general chain complexes}, we give a method for computing persistent homology and induced morphisms that applies to general chain complexes of persistence modules, relaxing the requirement of chains with infinite death times in the classic persistent homology algorithm.
See~\cite{skraba2013persistence,dey2024efficient,dey2014computing} for other approaches to computing persistent homology with general chain complexes.
In keeping with the theme of this paper, we will provide matrix-based procedures here -- we leave to future work the development of practical algorithms, say for filtrations of simplicial complexes, in which it may be best to avoid constructing the matrices explicitly.

\subsection{The classic persistent homology algorithm}\label{subsection: classic PH algorithm}

We review the classic matrix operations for computing persistent homology, making the evident translation to the language of this paper.
We mostly follow the approach of~\cite{zomorodian_carlsson_computing}, where a more thorough explanation can be found; also see~\cite{edelsbrunner_and_harer}.
Let $C$ be a chain complex consisting of persistence modules $C_n$ of finite type of the form
\[
C_n \cong \bigoplus_{j \in J_n} \mathbb{I}\interval{b_{n,j}, \infty},
\]
where $b_{n,j} < \infty$ in each case.
Such chain complexes arise from increasing filtrations of finite simplicial complexes, for instance.
Both the ordering of $J_n$ by birth times and the reverse will be useful: we will state the orders when they are used.
Because of the infinite death times, all entries of the matrix representing the boundary map $\partial_n \colon C_n \to C_{n-1}$ are nontrivial; we also let $\partial_n$ denote this matrix. 
Furthermore, the infinite death times imply that row and column operations can be interpreted as changes of decomposition when they are performed in the right order: an invertible matrix that is upper triangular when ordered by birth times represents an automorphism of $C_n$.

The barcode of the persistent homology module $H_n = \ker \partial_n / \im \partial_{n+1}$ can be computed using appropriate changes of decomposition.
These are summarized in the diagram below, which is best viewed ordering $C_{n+1}$ by birth times and $C_n$ in reverse by birth times.
With these orderings, $T_{n}$ and $S_{n}$ are lower triangular, whereas $T_{n+1}$ is upper triangular (the order of $C_{n+1}$ will be reversed in the next analogous diagram for $\partial_{n+2}$, making $T_{n+1}$ lower triangular with respect to the reversed order). 
Thus, $T_n$, $S_n$, and $T_{n+1}$ can be understood as changes of decomposition.
The matrices $S_n$ and $T_{n+1}$ will be defined by fully reducing $M_{n+1}$, so the homology will be read off from the pivots of the reduced matrix $D_{n+1}$.

\[
\begin{tikzcd}
C_{n}                                        & C_{n+1} \arrow[l, "\partial_{n+1}"'] \arrow[ld, "M_{n+1}"] \\
C_{n} \arrow[u, "T_{n}"] \arrow[d, "S_{n}"'] & C_{n+1} \arrow[ld, "D_{n+1}"] \arrow[u, "T_{n+1}"']        \\
C_{n}                                        &                                                           
\end{tikzcd}
\]

The computations are performed inductively in increasing dimension, beginning by letting $M_0$ be a zero matrix and $T_0$ be the identity.
Inductively, suppose $D_n = S_{n-1} M_n T_{n}$ was obtained by ordered row and column reduction of $M_n$ with its columns ordered by birth times and its rows in reverse by birth times.
Let $P_{n}$ be the set of $j \in J_n$ such that column $j$ of $D_n$ contains a pivot.
Then $M_{n+1} = T_n^{-1}\partial_{n+1}$ can be found by changing the rows of $\partial_{n+1}$ indexed by $P_{n}$ to zero (analogously to \cite[Lemma~4.2]{zomorodian_carlsson_computing}).
This is proved by noting 1) the row operations required to left multiply by $T^{-1}_{n}$ never alter rows indexed by $J_n \setminus P_{n}$ and 2) the rows of $M_{n+1}$ indexed by $P_{n}$ must be zero since the columns of $M_n T_n$ indexed by $P
_n$ are linearly independent and $M_n T_n M_{n+1} = T_{n-1}^{-1} \partial_n \partial_{n+1} = 0$. 
In practice, these rows of zeros can simply be deleted.
Order the columns of $M_{n+1}$ by birth times and the rows in reverse by birth times, and perform ordered row and column reduction on $M_{n+1}$ so that $D_{n+1} = S_{n} M_{n+1} T_{n+1}$ has pivots as its only nonzero entries.
The barcode of $H_n$ then has a bar $\interval{b_{n,j}, \infty}$ for each $j \in J_n \setminus P_n$ such that row $j$ of $D_{n+1}$ has no pivot and a bar $\interval{b_{n,j}, b_{n+1, p(j)}}$ for each $j \in J_n \setminus P_n$ such that $D_{n+1}$ has a pivot in position $j,p(j)$, where in the latter case the bar may be empty.
In $\ker \partial_n$, these bars correspond to the columns of $T_n S_n^{-1}$ indexed by $J_n \setminus P_n$, so these columns provide representative cycles for the persistent homology module.
If representative cycles are not needed, then row reduction of $M_{n+1}$ can be skipped, as the pivots of $D_{n+1}$ can be found after just column reduction: the pivot positions are all that are needed to find the barcode of $H_n$ and to perform the next inductive step.

\subsection{Comparison of the persistent homology algorithm and algebraic stability}\label{subsection: comparison with algebraic stability}

The main step in the computation of persistent homology barcodes above is column reduction of a matrix, and the key idea in our proof of algebraic stability, Theorem~\ref{theorem: algebraic stability finite modules}, is also column reduction.
The similarity can be explained by Theorem~\ref{theorem: barcodes of images}(\ref{item: image barcode 2}), which gives the algebraic meaning of column reduction.

In the proof of algebraic stability, $X$ is column reduced to give $XT$.
Theorem~\ref{theorem: barcodes of images} (\ref{item: image barcode 2}) shows each bar in the domain corresponding to a nontrivial pivot gives birth to a bar in the image, and the inclusion of the image into the codomain is represented by the pivot columns of $XT$.
We composed with the other morphism of the interleaving to verify the necessary bars of the image, those that came from sufficiently long bars of the domain, survive long enough.
Since Theorem~\ref{theorem: barcodes of images} shows the bars of the image are truncated bars of the codomain, this realizes the intuition of Figure~\ref{fig: interleaving} on all long enough bars of the domain simultaneously.
Specifically, the column-reduced form of $XT$ allows each bar of the image to be matched to the latest-dying bar in the codomain in which it has a nonzero component.
Intuitively, each time a new bar is introduced to the image, it must ``see'' a part of the codomain not seen by previous bars, and it is matched with the optimal bar of the codomain it can see.

For the computation of persistent homology, the column reduction serves to compute both an image and a kernel. 
In the steps above, we have already column reduced $M_n$, and $\ker \partial_n$ corresponds to the columns of the reduced matrix $M_n T_n$ indexed by $J_n \setminus P_n$.
To compute $H_n$, we must take the quotient by the image of the boundary map $\partial_{n+1} \colon C_{n+1} \to \ker \partial_n$.
Instead of using Theorem~\ref{theorem: barcodes of cokernels}, the particular setting of chains with infinite death times allows us to compute the barcode of the quotient directly from the image.
Theorem~\ref{theorem: barcodes of images} (\ref{item: image barcode 2}) shows the image can be computed by column reducing $M_{n+1}$, leading to the description of $H_n$ in terms of pivots of the reduced matrix $M_{n+1}T_{n+1}$.
By the ordering of bars, each time a new bar is introduced into the image, on homology, it leads to the death of the youngest bar it can see; this viewpoint of processing one bar at a time agrees with the description of the persistent homology algorithm in terms of pairing~\cite{edelsbrunner2002topological,edelsbrunner_and_harer}.

Thus, both applications of column reduction can be understood in terms of images, and in both cases, pivots are used to match bars of the image with bars of the codomain based on orderings of birth and death times.

\subsection{Induced maps in the classic persistent homology algorithm}\label{subsection: induced maps for classic PH algorithm}

Using the notation of Section~\ref{subsection: classic PH algorithm}, we now move on to the computation of induced maps on homology.
Suppose we additionally have a map of chain complexes $C \to C'$, with each map $C_n \to C'_n$ represented by a matrix $X_n$, and let $J'_n$, $P'_n$, $T'_n$, $S'_n$, and $M'_n$ be defined analogously to $J_n$, $P_n$, $T_n$, $S_n$, and $M_n$ above.
The induced map on homology $H_n \to H'_n$ is computed using the same changes of decomposition, as in the diagram.
\[
\begin{tikzcd}
C_n \arrow[r, "X_n"]                   & C'_n                                      \\
C_n \arrow[u, "T_n"] \arrow[d, "S_n"'] & C'_n \arrow[u, "T'_n"'] \arrow[d, "S'_n"] \\
C_n \arrow[r, dashed]                  & C'_n                                     
\end{tikzcd}
\]
Since the generating cycles correspond to the bars indexed by $J_n \setminus P_n$ and $J'_n \setminus P'_n$ in the lower copies of $C_n$ and $C'_n$, the map on homology is represented by the submatrix of $S'_n (T'_n)^{-1} X_n T_n (S_n)^{-1}$ consisting of columns indexed by $J_n \setminus P_n$ and rows indexed by $J'_n \setminus P'_n$.
Note that this representation includes any empty bars produced in the computation of barcodes described above.

Beginning with $X_n$, the matrix products can be computed by the same row and column operations used to compute $H_n$ and $H'_n$, and the order can be chosen so the reductions are performed at the same time as they are for $H_n$ and $H'_n$.
The following steps show the computation in which the operations to reduce $M'_{n+1}$ are used before those for $M_{n+1}$, which is the natural order to use when $H'_n$ is computed before $H_n$. 
\begin{enumerate}
    \item\label{item: classic PH induced map step 1} Delete the rows of $X_n$ indexed by $P'_n$ and denote the remaining matrix by $X'_n$.
    \item\label{item: classic PH induced map step 2} Perform the column operations used to reduce $M_n$ on $X'_n$, then delete the columns indexed by $P_n$.  Let the resulting matrix be $X''_n$.
    \item\label{item: classic PH induced map step 3} Perform the row operations used to reduce $M'_{n+1}$ on $X''_n$.
    Let the resulting matrix be $X'''_n$.
    \item\label{item: classic PH induced map step 4} For all row operations used to reduce $M_{n+1}$, perform the inverse column operations on $X'''_n$ (that is, when $\lambda$ times row $j_1$ is added to row $j_2$ in the reduction of $M_{n+1}$, subtract $\lambda$ times column $j_2$ from column $j_1$).
    The resulting matrix represents the induced map on homology in terms of the barcodes computed for $H_n$ and $H'_n$.
\end{enumerate}

After Step~\ref{item: classic PH induced map step 1}, $X'_n$ is the submatrix of $(T'_n)^{-1} X_n$ consisting of the rows indexed by $J'_n \setminus P'_n$, since, as in the classic persistent homology algorithm, the row operations required to compute left multiplication by $(T'_n)^{-1}$ never alter rows indexed by $J'_n \setminus P'_n$.
Step~\ref{item: classic PH induced map step 2} computes right multiplication by $T_n$, so $X''_n$ is the submatrix of $(T'_n)^{-1} X_n T_n$ consisting of the rows indexed by $J'_n \setminus P'_n$ and the columns indexed by $J_n \setminus P_n$.
Step~\ref{item: classic PH induced map step 3} computes left multiplication by $S'_n$ and can be performed on $X''_n$ since the operations only involve the rows indexed by $J'_n \setminus P'_n$; thus, $X'''_n$ is the submatrix of $S'_n (T'_n)^{-1} X_n T_n$ indexed by $J'_n \setminus P'_n$ and $J_n \setminus P_n$.
Similarly, Step~\ref{item: classic PH induced map step 4} computes right multiplication by $(S_n)^{-1}$, so the resulting matrix is the submatrix of $S'_n (T'_n)^{-1} X_n T_n (S_n)^{-1}$ indexed by $J'_n \setminus P'_n$ and $J_n \setminus P_n$, as required.
Due to the matrix reductions, the method has time complexity $O(m^3)$, where $m$ is the size of the larger of the index sets $J_n$ and $J'_n$.
Heuristically, since the matrix operations are reused from the computations of the barcodes, the computation of the map on homology can be expected to take about the same amount of time as the computation of the two barcodes and their representative cycles.

\subsection{Computation of persistent homology with general chain complexes}\label{subsection: computation with general chain complexes}

We conclude with a method for computing persistent homology and induced maps in a more general setting.
The representations of kernels and cokernels in Section~\ref{section: computations} can be used to compute persistent homology even when the chain modules have bars with finite death times.
That is, we now start with an arbitrary chain complex $C$ of persistence modules of finite type with given interval decompositions:
$C_n \cong \bigoplus_{j \in J_n} \mathbb{I}\interval{b_{n,j}, d_{n,j}}$.
Such a complex can arise from a filtration of simplicial complexes that is not necessarily increasing, that is, an $R$-indexed diagram of simplicial complexes with no requirement that structure maps are inclusions.

The computation is summarized in the diagram below, analogous to the one for the simpler case in Section~\ref{subsection: classic PH algorithm}.
Beginning with the boundary map $\partial_n$ and setting $Z_n = \ker \partial_n$, we aim to compute the inclusion $Z_n \hookrightarrow C_n$, the following boundary map with codomain restricted $C_{n+1} \to Z_n$, and finally the cokernel $H_n$ of this map, along with the quotient map $Z_n \to H_n$.
\[
\begin{tikzcd}
C_{n}                                                              & C_{n+1} \arrow[l, "\partial_{n+1}"'] \arrow[ld, "\widetilde{M}_{n+1}"] \\
Z_n \arrow[u, "\widetilde{T}_{n}"] \arrow[d, "\widetilde{S}_{n}"'] &                                                                        \\
H_n                                                                &                                                                       
\end{tikzcd}
\]
Applying Theorem~\ref{theorem: barcodes of kernels} to $\partial_n$ gives $Z_n$ and a matrix $\widetilde{T}_{n}$ representing the inclusion $Z_n \hookrightarrow C_n$.
Theorem~\ref{theorem: factor through ker and coker} then shows that $\widetilde{M}_{n+1} = (\widetilde{T}_n)^{-1} \partial_{n+1}$ represents the boundary map with restricted codomain.
Applying Theorem~\ref{theorem: barcodes of cokernels} to $\widetilde{M}_{n+1}$ gives $H_n$ and a matrix $\widetilde{S}_n$ representing the quotient map $Z_n \to H_n$.
Recall that the direct sums in Theorems~\ref{theorem: barcodes of kernels} and~\ref{theorem: barcodes of cokernels} may include empty bars, so these are included here as well.

The matrix products can be computed by row and column operations, although they are more extensive than before.
The computation of kernels and cokernels by Theorems~\ref{theorem: barcodes of kernels} and~\ref{theorem: barcodes of cokernels} require two matrix reductions each.
For the kernel, applying Theorem~\ref{theorem: barcodes of kernels} will first column reduce $\partial_n$ to obtain $\partial_n T_n$, and $T_n$ is recorded by applying the same operations to the identity matrix.
Then $\widetilde{T}_n$ is found by reordering $T_n$ as in the theorem and column reducing again.
To compute $\widetilde{M}_{n+1} = (\widetilde{T}_n)^{-1} \partial_{n+1}$, we apply the inverse row operations of both reductions to $\partial_{n+1}$.
Similarly, $\widetilde{S}_n$ is computed in two row reductions, and the inverse column operations can be used to compute $\widetilde{T}_n \widetilde{S}_n^{-1}$ if representative cycles are required.
Note that $\widetilde{T}_n \widetilde{S}_n^{-1}$ does not generally represent a morphism $H_n \to C_n$; it can instead be used to represent a morphism $\overline{H}_n \to C_n$, where $\overline{H}_n$ is defined by extending the bars of $H_n$ to have infinite death times.
The time complexity of these computations is again cubic in the size of the chain modules.
They do, however, require a greater number of matrix reductions than the classic persistent homology algorithm and should be expected to be proportionally slower.

Finally, we check that induced maps on homology can also be computed in this setting.
Suppose again that we have another chain complex $C'$, a map $C \to C'$ with each $C_n \to C'_n$ represented by a matrix $X_n$, and analogous matrices $\widetilde{T}'_n$ and $\widetilde{S}'_n$.
\[
\begin{tikzcd}
C_{n} \arrow[r, "X_n"]                                                               & C'_{n}                                                            \\
Z_n \arrow[u, "\widetilde{T}_{n}"] \arrow[d, "\widetilde{S}_{n}"'] \arrow[r, dashed] & Z'_n \arrow[d, "\widetilde{S}'_n"] \arrow[u, "\widetilde{T}'_n"'] \\
H_n \arrow[r, dashed]                                                                & H'_n                                                             
\end{tikzcd}
\]
Applying Theorem~\ref{theorem: factor through ker and coker} to the upper square shows $(\widetilde{T}'_n)^{-1} X_n \widetilde{T}_n$ represents the induced map on cycles, then applying it to the lower square shows $\widetilde{S}'_n (\widetilde{T}'_n)^{-1} X_n \widetilde{T}_n (\widetilde{S}_n)^{-1}$ represents the induced map on homology.
As before, the matrix products can be computed by reusing the row and column operations from the computations of $H_n$ and $H'_n$.

\newpage

\bibliographystyle{unsrt}
\bibliography{bibliography.bib}

\end{document}